\documentclass[11pt,a4paper,openany]{amsart}
\usepackage{tikz}
\usepackage{amsmath,amsfonts}
\usepackage{latexsym}
\usepackage{amssymb}
\usepackage{color}
\usepackage{epsfig}
\usepackage{subfigure}
\usepackage{mathrsfs}
\usepackage{amscd}
\usepackage{amsthm}
\usepackage{color}
\usepackage{cite}

\usepackage{graphicx}
\usepackage{float}

\numberwithin{equation}{section}
\newtheorem{proposition}{Proposition}[section]
\newtheorem{definition}{Definition}[section]
\newtheorem{lemma}{Lemma}[section]
\newtheorem{theorem}{Theorem}[section]
\newtheorem{corollary}{Corollary}[section]
\newtheorem{remark}{Remark}[section]

\newcommand\lambdabar{\bar{\lambda}}
\newcommand\specl{dE_{\mathbf{\sqrt{\LL_0}}}(\lambda)}
\newcommand{\R}{\mathbb{R}}

\newcommand{\rr}{\mathrm{r}}

\newcommand{\Z}{{\mathbb{Z}}}
\newcommand{\LL}{{\mathcal{L}}}
\newcommand{\CC}{{\mathcal{C}}}

\begin{document}
\title[Strichartz estimates on metric cones]
{Strichartz estimates and wave equation in a conic singular space}

\author{Junyong Zhang}
\address{Department of Mathematics, Beijing Institute of Technology, Beijing
100081, China; School of Mathematics,
Cardiff University, Cardiff CF24 4AG, United Kingdom}
\email{zhang\_junyong@bit.edu.cn; ZhangJ107@cardiff.ac.uk}

\author{Jiqiang Zheng}
\address{Institute of Applied Physics and Computational Mathematics, Beijing 100088}
\email{zheng\_jiqiang@iapcm.ac.cn; zhengjiqiang@gmail.com}

\maketitle

\begin{abstract} Consider the metric cone $X=C(Y)=(0,\infty)_r\times Y$ with metric $g=dr^2+r^2h$ where the cross section $Y$ is
a compact $(n-1)$-dimensional Riemannian manifold $(Y,h)$. Let $\Delta_g$ be
the positive Friedrichs extension Laplacian on $X$ and let $\Delta_h$ be the positive Laplacian on $Y$,  and consider the operator $\LL_V=\Delta_g+V_0 r^{-2}$ where
$V_0\in\CC^\infty(Y)$ such that $\Delta_h+V_0+(n-2)^2/4$ is a strictly positive operator on $L^2(Y)$.
In this paper, we prove global-in-time Strichartz estimates without loss regularity for the wave equation associated with the operator $\LL_V$. It verifies a conjecture in Wang\cite[Remark 2.4]{wang} for wave equation.
The range of the admissible pair is sharp and the range is influenced by the smallest eigenvalue of $\Delta_h+V_0+(n-2)^2/4$. To prove the result,
we show a Sobolev inequality and a boundedness of a generalized Riesz transform in this setting. In addition,
as an application, we study the well-posed theory and scattering theory for energy-critical wave equation with small data on
this setting of dimension $n\geq3$.

\end{abstract}

\begin{center}
 \begin{minipage}{120mm}
 \end{minipage}
 \end{center}

 \tableofcontents

\section{Introduction and main result}

\subsection{Background: the setting and Strichartz estimate}Suppose that $(Y,h)$ is a compact $(n-1)$-dimensional Riemannian manifold, the metric
cone $X=C(Y)$ on $Y$ is the space $(0,\infty)_r\times Y$ with $g=dr^2+r^2h$. The metric cone $X$ has a simplest geometry singularity and it has incomplete metric.
One can complete it to $C^*(Y)=C(Y)\cup P$ where $P$ is its cone tip. Denote $\Delta_g$ the Friedrichs extension of Laplace-Beltrami from the domain
$\CC_c^\infty(X^\circ)$, compactly supported smooth functions on the
interior of the metric cone. There is a number of works to extend the theory of the Laplace operator $\Delta_g$ on smooth manifolds to certain Riemannian spaces with
such conical singularities;  for example the spectral theory, see Cheeger \cite{C1,C2}.\vspace{0.2cm}

Solutions to the wave equation on cones and related spaces were studied
from the perspective of wave diffraction from the cone point, see \cite{Som,Frie,Frie1}.
In the setting of exact cones, Cheeger and Taylor \cite{CT,CT1} studied the Laplacian from points of the functional calculus. Melrose and Wunsch \cite{MW} proved
a propagation of singularities property for solutions to the wave equation on the more general setting of conic manifolds. In addition,
the other aspects of Schr\"odinger operator on the metric cone, even with
potentials that are homogeneous of degree $-2$, also have
been extensively studied. For instance, the asymptotical behavior of Schr\"odinger propagator was considered in \cite{wang}. The heat kernel was studied in \cite{Mo,L2} and Riesz transform
kernel was investigated in \cite{HL,L1}.  The $L^p$-estimates were studied in \cite{L} and the restriction estimate for Schr\"odinger solution was studied 
by the first author \cite{Z1}. The Strichartz estimates for Schr\"odinger were proved on a flat cone by Ford \cite{Ford}, on
polygonal domains by Blair-Ford-Herr-Marzuola\cite{BFHM}, on exterior polygonal domains by Baskin-Marzuola-Wunsch \cite{BMW}, and on the metric cone
by the authors \cite{ZZ1}. Regarding the Strichartz estimate for wave equation on cones,  Blair-Ford-Marzuola \cite{BFM} have established the Strichartz
inequalities on a flat cone of dimension two, that is, $Y=\mathbb{S}^1_\rho$. However, one needs the explicit form of wave propagator when $Y=\mathbb{S}^1_\rho$ in \cite{BFM},
hence the methods of \cite{BFM} can not be applied to our general setting. \vspace{0.2cm}

In this paper, we prove the Strichartz estimates for the solution to
wave equations on metric cone and, as an application, we study the well-posed theory and scattering theory for the energy-critical nonlinear wave equation.
Here we recall the Schr\"odinger operator $\mathcal{L}_V=\Delta_g+V$ where $V=V_0(y) r^{-2}$ and $V_0(y)$ is a smooth function on the section $Y$. Our motivation to study this Schr\"odinger operator is to understand the regularity or singularity of wave propagates in a singular black hole. For example,
we refer to \cite{DSS, RW} for the connection with Schwarzchild black hole or \cite{Mon, Zer} for the Reissner-Nordst\"om black hole.  With respect to the potential, since the decay
of the inverse-square potential is closely related to the angular momentum as $r\to\infty$,
we are known that inverse square decay of the potential is in some
sense critical for the spectral and scattering theory. In context of this paper, we remark here that the inverse-square type potential is homogeneous of degree $-2$ and is at the boardline of decay in order to guarantee validity of
Strichartz estimate; see Goldberg-Vega-Visciglia \cite{GVV}. The property of the inverse-square type potential near the cone tip, or near infinity-end, or both, brings the singular behavior, however, it is a natural potential.  For example \cite{Ca},
 the Dirac equation with a Coulomb potential can be recast in the form of a Klein-Gordon equation with an inverse-square type potential.
 \vspace{0.2cm}

Consider the solution $u: I\times X\rightarrow \R$ to the initial value
problem (IVP) for the wave equation on metric cone $X$,
\begin{equation}\label{equ:wave}
\begin{cases}
\partial_{t}^2u+\LL_V u=F(t,z), \quad (t,z)\in I\times X; \\ u(0)=u_0(z),
~\partial_tu(0)=u_1(z).
\end{cases}
\end{equation}
It is well-known that the Strichartz estimate implies the decay and regularity of the solutions to
the wave equations, and plays an important role in the studying of nonlinear wave equations.
More precisely, let $u$ be the solution to \eqref{equ:wave} and the time interval $I\subseteq\R$, the Strichartz estimate states an inequality in the form of
\begin{equation}\label{stri}
\begin{split}
&\|u(t,z)\|_{L^q_t(I;L^{\rr}_z(X))}+\|u(t,z)\|_{C(I;\dot H^s(X))}\\
&\qquad\lesssim \|u_0\|_{\dot H^s(X)}+\|u_1\|_{\dot
H^{s-1}(X)}+\|F\|_{L^{\tilde{q}'}_t(I;L^{\tilde{\rr}'}_z(X))},
\end{split}
\end{equation}
where $\dot H^s=\LL_V^{-\frac s2} L^2(X)$ denotes the homogeneous $L^2$-Sobolev space over $X$ and the pairs $(q,\rr), (\tilde{q},\tilde{\rr})\in [2,\infty]\times [2,\infty)$
satisfy the wave-admissible condition
\begin{equation}\label{adm}
\frac{2}q+\frac{n-1}{\rr}\leq\frac{n-1}2,\quad (q,\rr,n)\neq(2,\infty,3)
\end{equation}
and the scaling condition
\begin{equation}\label{scaling}
\frac1q+\frac n\rr=\frac n2-s=\frac1{\tilde{q}'}+\frac
n{\tilde{\rr}'}-2.
\end{equation}
For $s\in\R$, we say the pair $(q,\rr)\in \Lambda_s$ if $(q,\rr)\in [2,\infty]\times [2,\infty)$ satisfies \eqref{adm} and \eqref{scaling}.  \vspace{0.2cm}

Due to the importance of the Strichartz inequalities, there is a lot of work studying the Strichartz inequalities on Euclidean space or manifolds; for example, see \cite{Str, GV, KT, ST} and references therein.
In the following, we in particular focus on recalling the most relevant work about the Strichartz estimate on a metric cone,
or on a slightly different setting of asymptotically conic manifold, or with a perturbation of inverse-square type potentials.  Our setting metric cone is close to the asymptotically conic manifold $M$ which, outside some compact set, is isometric to the conical space
$X$ away the cone tip. On the non-trapping asymptotically
conic manifold $M$, for Schr\"odinger equation, Hassell, Tao and Wunsch \cite{HTW, HTW1} and Mizutani \cite{Miz} showed the local-in-time Strichartz estimates;
the global-in-time Strichartz inequality including endpoint case was proved by Hassell and the first author in \cite{HZ} for Schr\"odinger and in \cite{Z2} for wave equation; and very recently
Bouclet and Mizutani \cite{BM} and the authors \cite{ZZ2} showed the global-in-time Strichartz estimates on asymptotically
conic manifold even with a hyperbolic trapped geodesic. As remarked above, the perturbation of the inverse-square potential
is nontrivial. In \cite{BPST, BPSS},  the additional perturbation of the inverse-square potential was taken into account when they studied the Euclidean standard Strichartz estimate for Schr\"odinger and wave, which is a tough task.
On a flat cone of dimension $2$,
Blair-Ford-Marzuola \cite{BFM} have established the Strichartz
inequalities for wave  by developing a representation of fundamental solution to the wave equation on the flat cone $C(\mathbb{S}^1_\rho)$ which is also applied to
the Schr\"odinger case in \cite{Ford}. \vspace{0.2cm}

\subsection{Main result and  the sketch of proof}
In our present general setting, we need to consider the propagator of the dispersive equation associated with the operator $\LL_V$ which is influenced by the geometry and the inverse-square type potential.
The authors \cite{ZZ1} proved the full
range of global-in-time Strichartz estimates for Schr\"odinger equation associated with the operator $\mathcal{L}_V$ which proved Wang's conjecture \cite[Remark 2.4]{wang}  for Schr\"odinger. \vspace{0.2cm}

In this vein (as in \cite{ZZ1}), we intend to prove the global-in-time Strichartz estimate  for wave equation associated with the operator $\mathcal{L}_V$, but with innovative aspects to combat difficulties arising from wave equation.
More precisely, we prove
the following results.

\begin{theorem}[Global-in-time Strichartz estimate]\label{thm:Strichartz} Assume that $(X, g)$ is a metric cone of dimension
$n\geq3$. Let $\LL_V=\Delta_g+V$ where $r^2V=:V_0(y)\in\CC^\infty(Y)$ such that $\Delta_h+V_0(y)+(n-2)^2/4$ is a strictly positive operator on $L^2(Y)$ and its smallest
eigenvalue is $\nu_0^2$ with $\nu_0>0$. Suppose
that $u$ is the solution of the Cauchy problem \eqref{equ:wave} with the initial data $u_0\in \dot H^{s}, u_1\in \dot H^{s-1}$ for $s\in\R$. \vspace{0.2cm}

$\mathrm{(i)}$ If $V\equiv0$, then the Strichartz estimate \eqref{stri} holds for all $(q,\rr)$, $(\tilde q, \tilde \rr)\in \Lambda_s$. \vspace{0.2cm}

$\mathrm{(ii)}$  If $V\not\equiv0$, then the Strichartz estimate \eqref{stri} holds for all $(q,\rr)$, $(\tilde q, \tilde \rr)\in \Lambda_{s,\nu_0}$
where
\begin{equation}\label{Ls}
\Lambda_{s,\nu_0}=\{(q,\rr)\in\Lambda_s: 1/\rr>1/2-(1+\nu_0)/n \}.
\end{equation}

\end{theorem}

\begin{remark} From the first result $\mathrm{(i)}$, the geometry of metric cone,  possibly having conjugated points, does not influence the Strichartz estimate even though the conjugated points
cause the failure of dispersive estimate.
From the restriction \eqref{Ls}, the Strichartz estimate  is affected by the positive square root of the smallest eigenvalue  of $\Delta_h+V_0(y)+(n-2)^2/4$. The requirement \eqref{Ls} is sharp, see Subsection 6.3.
 \end{remark}


\begin{remark} The set $\Lambda_{s,\nu_0}$ makes sense when $s\in[0,1+\nu_0)$ otherwise it is empty.  Compared with $\Lambda_s$,  one can check that
$\Lambda_{s,\nu_0}=\Lambda_s$ for $s\in [0,1/2+\nu_0)$, and while $\Lambda_{s,\nu_0}\subset \Lambda_s$ for $s\in [1/2+\nu_0, 1+\nu_0)$.
In particular, $V\equiv0$, hence $\nu_0> (n-2)/2$ large enough so that $\Lambda_{s,\nu_0}=\Lambda_s$ for $s\in\R$,  thus the second conclusion is consistent with the first one.
 \end{remark}

\begin{remark} If $\nu_0>\frac1{n-1}$, the Strichartz estimates hold for $(q,\rr)$ such that $(\frac1q,\frac1\rr)$ belongs to the region ABCEF
when $n\geq4$ and ABO when $n=3$.
Compared with the Euclidean case, the Strichartz estimate fails in the region CDOE of Figure 1.
If $0<\nu_0<\frac1{n-1}$, then the line EC is above the line FB, we do not have the Strichartz estimate with
$q=2$. The result illustrates that the smallest eigenvalue  of $\Delta_h+V_0(y)+(n-2)^2/4$
plays an important role in the Strichartz estimate.

\end{remark}

\begin{center}
 \begin{tikzpicture}[scale=1]
\draw[->] (0,0) -- (4,0) node[anchor=north] {$\frac{1}{q}$};
\draw[->] (0,0) -- (0,4)  node[anchor=east] {$\frac{1}{\rr}$};
\draw (0,0) node[anchor=north] {O}
(3,0) node[anchor=north] {$\frac12$};
\draw  (0, 3) node[anchor=east] {$\frac12$}
       (0, 0.6) node[anchor=east] {$\frac12-\frac{1+\nu_0}{n}$}
       (0, 1.2) node[anchor=east] {$\frac12-\frac{1}{n-1}$};

\draw[thick] (3,0) -- (3,1.2)  
              (3,0.6) -- (0,0.6)
              (3,1.2) -- (0,3);

\draw[dashed,thick] (0,1.2) -- (3,1.2) 
                    (3,0.6) --  (0,1.8)
                    (3,1.2) -- (0,2.4);
\draw (-0.1,3.2) node[anchor=west] {A};
\draw (2.9,1.2) node[anchor=west] {B};
\draw (2.9,0.6) node[anchor=west] {C};
\draw (2.9,0.15) node[anchor=west] {D};
\draw (-0.1,0.4) node[anchor=west] {E};
\draw (-0.1,1.0) node[anchor=west] {F};
\draw (2.5,2) node[anchor=west] {$\frac{1}{q}+\frac{n}{\rr}=\frac{n}{2}-\frac{n+1}{2(n-1)}$};
\draw (3.4,0.9) node[anchor=west] {$\frac{1}{q}+\frac{n}{\rr}=\frac{n}{2}-(\frac12+\nu_0)$};
\draw (1.65,2.88) node[anchor=west] {$\frac2q+\frac{n-1}{r}=\frac{n-1}{2}$};

\draw[<-] (1.6,2.1) -- (2,2.6) node[anchor=south]{$~$};
\draw[<-] (2,1.65) -- (2.5,2) node[anchor=north]{$~$};
\draw[<-] (2.4,0.9) -- (3.4,0.9) node[anchor=north]{$~$};

\path (2,-1) node(caption){Fig 1. $n\geq4$};  

\draw[->] (8,0) -- (12,0) node[anchor=north] {$\frac{1}{q}$};
\draw[->] (8,0) -- (8,4)  node[anchor=east] {$\frac{1}{\rr}$};
\path (9.6,-1) node(caption){Fig 2. $n=3$};  

\draw  (8.1, -0.1) node[anchor=east] {O};
\draw  (11, 0) node[anchor=north] {$\frac12$};
\draw  (8, 3) node[anchor=east] {$\frac12$}
(8, 1) node[anchor=east] {$\frac16$};

\draw[thick] (8,3) -- (11,0);  
\draw[dashed,thick] (8,1) -- (11,0); 
\draw (7.9,3.15) node[anchor=west] {A};
\draw (10.9,0.2) node[anchor=west] {B};
\draw (7.9,1.15) node[anchor=west] {C};
\draw (10,2.6) node[anchor=west] {$\frac{2}{q}+\frac{2}{\rr}=1$};
\draw (10.7,1.3) node[anchor=west] {$\frac{1}{q}+\frac{3}{\rr}=\frac32-1$};

\draw (11,0) circle (0.06);

\draw[<-] (9,2.1) -- (10,2.6) node[anchor=south]{$~$};
\draw[<-] (10,0.5) -- (10.7,1.3) node[anchor=south]{$~$};

\path (6,-1.5) node(caption){Diagrammatic picture of the range of $(q,\rr)$, when $\nu_0>1/(n-1)$.};  

\end{tikzpicture}

\end{center}

\begin{remark} The restriction $1/\rr>1/2-(1+\nu_0)/n$ is also necessary for Schr\"odinger by a similar counterexample constructed in Subsection 6.3.  The reason for disappearance of this restriction in the Strichartz estimate of Schr\"odinger established by the authors \cite[Theorem 1.1]{ZZ1}
is that we only consider the estimate at regularity level  $s=0$. This is same to
the case here for wave $\Lambda_{s,\nu_0}=\Lambda_s$ for $s\in [0,1/2+\nu_0)$ in which the restriction $1/\rr>1/2-(1+\nu_0)/n$
disappears. The argument for wave needs more techniques on the Sobolev inequality and Riesz transform.
\end{remark}

\begin{remark}
Compared with the result involving the derivatives \cite[Theorem 9]{BPSS}, the result in Theorem \ref{thm:Strichartz} needs to consider  the influences of conical singular geometry and the potential $V_0(y)r^{-2}$ (rather than $V_0(y)\equiv c$).
 \end{remark}

We sketch the idea and argument of the proof here. The usual method to derive the Strichartz estimate is  Keel-Tao's \cite{KT} abstract method in which we need dispersive estimate and $L^2$-estimate.
In our setting, however, there are two difficulties to prevent us from obtaining the dispersive estimate. The first one arises from the conjugated points from the geometry, and the second one  from
the inverse-square potential. First, the degeneration of  projection between the conjugated points will slow down the dispersive decay estimate of the Schr\"odinger or wave propagator,
which was illustrated in \cite{HW, HZ}.  Second,  as discussed in \cite{HZ,ZZ1} for Schr\"odinger, it is not possible to obtain a dispersive estimate for
half wave operator $e^{i(t-s)\sqrt{\LL_V}}$ with norm $O(|t-s|^{-\frac{n-1}2})$ as $|t-s|\to \infty$ due to the influence of the negative inverse-square potential; see  \cite{BPSS, BPST} for
the perturbation of inverse-square potential on  Euclidean space. \vspace{0.2cm}

There are two key points, which have been established and used in \cite{ZZ2} for Schr\"odinger equation, to treat the two issues. The first one is to microlocalize the propagator which separates the conjugated points.
We achieve this through studying the property of the micro-localized spectral measure associated with the operator $\LL_0$, i.e. without potential. The second key one is to establish the global-in-time local smoothing estimate
which is proved via a variable separating argument.  \vspace{0.2cm}

More precisely, we first show (i) in Theorem \ref{thm:Strichartz} in which we do not need to consider the potential.
To obtain the Strichartz estimate for $\LL_0$, as  in \cite{HZ, Z2, ZZ2}, our strategy is Keel-Tao's abstract method where we need
the property of the micro-localized spectral measure to prove dispersive estimate and $L^2$-estimate.
However, we should modify the argument to adapt to wave equation and sharpen the Strichartz estimate in a Lorentz space.
Compared with the Schr\"odinger, the wave propagator multiplier $e^{it\lambda}$ is less oscillation than the Schr\"odinger's $e^{it\lambda^2}$, thus we need
a Littlewood-Paley theory  in our setting, in particular the Littlewood-Paley square function inequality on Lorentz space. The key is
to show a Mikhlin-H\"ormander multiplier theorem. We notice that our setting is a measure space in which the wave operator has finite propagation speed and one has doubling
condition. Thus, from Chen-Ouhabaz-Sikora-Yan \cite{COSY}, the multiplier estimate on $L^p$ is a consequence of a spectral measure estimate which can be obtained
from the property of micro-localized spectral measure and $TT^*$-method.
The Littlewood-Paley (LP) square function inequality on Lorentz space is finally obtained from the interpolation characteristic of Lorentz space.
Once the LP square function estimate has been established, we may assume that the initial data is frequency localized in $\{\lambda\sim 2^k\}$. The argument \cite{Z2} can be modified to
prove the Strichartz estimate. We remark that the property of  microlocalized spectral measure capturing the figures of the decay and oscillation behavior which  plays an important role in this part.\vspace{0.2cm}

Next we show (ii) in Theorem \ref{thm:Strichartz}.  We use a perturbation method \cite{JSS,RS} to derive (ii) from a local smoothing and the results of (i).
The usual way to  show a local smoothing estimate  is through establishing the resolvent estimate for $\LL_V$ at low and high frequency. Unlike the usual way,
 we  avoid the resolvent estimate to show the global-in-time local smoothing estimate  by using the explicit formulas with separating variables expression.
In addition, in particular for obtaining Strichartz estimate at $q=2$, we need a  double end-points inhomogeneous Strichartz estimate for $\LL_0$ which is not proved in (i). To this end, we modify an argument in \cite{HZ}
to adapt to wave equation. Another difference between wave and Schr\"odinger should not be ignored, that is,
wave's double end-points inhomogeneous Strichartz estimate involves some negative derivative. This requires us to  study the $L^p$-boundedness theory of  a generalized Riesz transform operator $\Delta_g^{s/2}\LL_V^{-s/2}$.
It is worth mentioning that the method  of \cite{BPSS} is based on the fact the potential $V_0(y)r^{-2}=cr^{-2}$ independent of $y$, and the method can not
be directly used for the potential in our setting.  To obtain our result, we have to resort to the harmonic analysis tools, such as the Sobolev inequality,  associated with $\LL_V$
which are established in our preliminary sections.
\vspace{0.2cm}

\subsection{Application: energy-critical wave equation}As an application of the global-in-time Strichartz estimates, we study the nonlinear wave equation on $X$ of dimension $n\geq3$
\begin{equation}\label{equ:energy-critical}
\begin{cases}
\partial_{tt} u+\LL_V u+\gamma|u|^{\frac{4}{n-2}}u=0,\qquad
(t,z)\in\R\times X, \\
(u(0),\partial_t u(0)=(u_0(z), u_1(z))\in \dot H^1(X)\times L^2(X),\qquad z\in X.
\end{cases}
\end{equation}
where $\dot H^1(X)=\LL_V^{-\frac12} L^2(X)$ is the homogeneous Sobolev space over $X$  and $\gamma=\pm1$
which corresponds to the defocusing and focusing case respectively. Notice that our metric cone $X$ is invariant under the dilation variable change, hence our equation model
has symmetries of time translation and scaling dilation but not translation invariant in space. The class of solutions to \eqref{equ:energy-critical} is invariant by the
scaling
\begin{equation}\label{scale}
(u,u_t)(t,z)\mapsto \big(\lambda^{\frac{n-2}2} u(\lambda t,\lambda
z),\lambda^{\frac n2}u_t(\lambda t,\lambda z)\big), ~~\forall~\lambda>0.
\end{equation}
One can check that the only homogeneous $L^2$-based Sobolev space  $
\dot H^{1}(X)\times L^2(X)$ is invariant under \eqref{scale}.  The rescaling also remains invariant for the energy of solutions defined by
\begin{equation}\label{energy}
E(u,u_t)=\frac12\int_{X}\big(|\partial_tu|^2+|\LL_V
u|^2\big) d\mu+\frac{\gamma(n-2)}{2n}\int_X |u|^{\frac{2n}{n-2}}d\mu,
\end{equation} which is a
conserved quantity for \eqref{equ:energy-critical} and  where $d\mu=\sqrt{|g|}dz=r^{n-1}dr dh$. Hence the Cauchy problem \eqref{equ:energy-critical} falls in the class of
\emph{energy-critical} problem.
Because of the conserved quantities at the critical
regularity, the energy-critical equations have been the most extensively
studied instances of NLW. In the Euclidean space, that is $X=\R^n$ and $V=0$, for the defocusing energy-critical NLW, it
has been known now that the
solutions that are global and scatter when the initial data is in $\dot H^1\times L^2$  which could be arbitrarily large, see Grillakis
\cite{Gri90}, Kapitanski \cite{Kapi94},
Shatah and Struwe\cite{ShaStr94}, Bahouri and G\'erard \cite{BG}, Tao
\cite{T07} and the references therein. For the focusing energy-critical  NLW in dimensions $ n\in\{3,4,
5\}$, Kenig and Merle \cite{KM1} obtained the
dichotomy-type result under the assumption that $E (u_0, u_1) < E
(W, 0)$, where $W$ denotes the ground state of an nonlinear elliptic equation.  From this, it is not an easy thing to study the global existence and scattering theory of the initial value problem
with large data in $\dot H^1\times L^2$ even though in the Euclidean space.\vspace{0.2cm}

In this paper, as an application of Strichartz estimate, we study the global existence and scattering for the Cauchy problem \eqref{equ:energy-critical} with initial data in $\dot H^1\times L^2$ but small enough.
Our result for the energy-critical wave equation is the following.

\begin{theorem}\label{thm:NLW} Let $X$ be
metric cone of dimension $n\geq 3$ and $\LL_V=\Delta_g+V$ as in Theorem \ref{thm:Strichartz}. Let $\gamma=\pm1$ and suppose that the initial data $(u_0, u_1)\in \dot H^1(X)\times L^2(X)$.
Assume the above $\nu_0>1/2$. Then
there exists $T=T(\|(u_0, u_1)\|_{H^1(X)\times L^2(X)})>0$ such that the energy-critical equation \eqref{equ:energy-critical} is local wellposed in $I=[0,T)$ and
the unique solution $u$ obeys
\begin{equation}\label{local-b}
u\in C(I; \dot H^1(X))\cap L_t^{q}(I; L^\rr(X)),\quad I=[0,T)
\end{equation}
where
\begin{equation}\label{q-r}
(q,\rr)=\begin{cases} ((n+2)/(n-2),2(n+2)/(n-2)),\quad 3\leq n\le 6;\\
(2, 2n/(n-3)),\quad n\geq 7.
\end{cases}
\end{equation}
In addition, if $\|(u_0, u_1)\|_{\dot H^1(X)\times L^2(X)}\leq \delta$ for a small enough constant $\delta$, there exists a global solution $u$ and the solution $u$ scatters in the sense that
there are $(u_0^{\pm}, u_1^\pm)\in \dot H^1(X)\times L^2(X)$ such that
\begin{equation}\label{1.2.1}
\bigg\|{u(t)\choose \dot{u}(t)}-V_0(t){u_0^\pm\choose
u_1^\pm}\bigg\|_{\dot{H}^{1}_z\times L^{2}_z}
\longrightarrow 0,\quad \text{as}\quad t\longrightarrow \pm\infty.
\end{equation}
where
\begin{equation}\label{v0tdefin}
V_0(t) = {\dot{K}(t)\quad K(t) \choose \ddot{K}(t)\quad \dot{K}(t)},
\quad K(t)=\frac{\sin(t\sqrt{\LL_V})}{\sqrt{\LL_V}}.\end{equation}
\end{theorem}
\begin{remark}
This result is  similar to the well known result for energy-critical wave equation in Euclidean space and the global existence and scattering theory for small
data still holds on the metric cone manifold. Like the Euclidean result, this small initial result  is also a cornerstone result for future work with large data on this setting.
The assumption on $\nu_0>1/2$ guaranteeing that the Strichartz estimate holds for all $(q,\rr)\in \Lambda_s$ with $s=1$  can be improved, we do not pursue this here.
\end{remark}
We prove this result by using Picard iteration argument which was used in Euclidean space, see Tao's book \cite{Tbook}. The key ingredient is the global-in-time
Strichartz estimate in Theorem \ref{thm:Strichartz}.
\vspace{0.2cm}

Finally, we introduce some notation. We use $A\lesssim B$ to denote
$A\leq CB$ for some large constant C which may vary from line to
line and depend on various parameters, and similarly we use $A\ll B$
to denote $A\leq C^{-1} B$. We employ $A\sim B$ when $A\lesssim
B\lesssim A$. If the constant $C$ depends on a special parameter
other than the above, we shall denote it explicitly by subscripts.
For instance, $C_\epsilon$ should be understood as a positive
constant not only depending on $p, q, n$, and $M$, but also on
$\epsilon$. Throughout this paper, pairs of conjugate indices are
written as $p, p'$, where $\frac{1}p+\frac1{p'}=1$ with $1\leq
p\leq\infty$. We denote $a_\pm$ to be any
quantity of the form $a\pm\epsilon$ for any small $\epsilon>0$.
\vspace{0.2cm}

This paper is organized as follows: In Section 2, we recall and prove some analysis results such as the spectral measure and
the Littlewood-Paley theory in our setting. Section 3 is devoted to the Sobolev inequality and a generalized Riesz transform.  In Section 4, we prove our
main Theorem \ref{thm:Strichartz} on Strichartz esimate for wave with $\LL_0$. We prove a double endpoint  inhomogeneous  Strichartz estimate in Section 5.
In the section 6, we show a local smoothing estimate and prove Theorem \ref{thm:Strichartz} for wave with $\LL_V$. We construct a counterexample to show the sharpness. In the final section, we utilize the Strichartz
estimates to show Theorem \ref{thm:NLW}.\vspace{0.2cm}

{\bf Acknowledgments:}\quad  The authors would like to thank Andras Vasy and Andrew Hassell for their
helpful discussions and encouragement. The first author is grateful for the hospitality of Stanford University where the project was initiated.
J. Zhang was supported by National Natural
Science Foundation of China (11771041,11831004) and H2020-MSCA-IF-2017(790623), and J. Zheng was partially supported by the European
Research Council, ERC-2014-CoG, project number 646650 Singwave and ANR-16-TERC-0006-01:ANADEL.\vspace{0.2cm}

\section{Some analysis results associated with the operator $\LL_V$}
This paper is devoted to the wave equation associated with the operator $\LL_V$, hence we need extra harmonic analysis tools which are influenced by the geometry of
the cone $X$ and the potential $V$,  even though some ones have been established in previous work\cite{Z1,ZZ1}.  The purpose of this section is to show and recall the analysis tools
for usage in the following sections.

\subsection{Basic harmonic analysis tools and notation on the metric cone} Recall that the metric cone $X=C(Y)=(0,\infty)_r\times Y$ is equipped with the metric $g=dr^2+r^2h$ and the cross section $Y$ is
a compact $(n-1)$-dimensional Riemannian manifold $(Y,h)$. Let  $z=(r, y)\in
\R_+\times Y$, then the measure on $C(Y)$ is
\begin{equation}\label{meas}
dg(z)=d\mu(z)=r^{n-1}dr dh=r^{n-1}dr d\mu_Y(y).
\end{equation}
For $1\leq p<\infty$,  define the $L^p(X)$ space by the complement of $\mathcal{C}_0^\infty(X)$ under the norm
\begin{equation}\label{Lp}
\|f\|_{L^p(X)}^p=\int_{X}|f(r,y)|^p d\mu(r,y)=\int_0^\infty\int_Y  |f(r,y)|^p r^{n-1}dr
d\mu_Y(y).\end{equation}
Let $d$ (resp. $d_Y$) be the distance function on $X=C(Y)$ (resp. $Y$) then, for instance see \cite{CT}, the distance on a metric cone is
\begin{equation}
d(z,z')=\begin{cases}\sqrt{r^2+r'^2-2rr'\cos(d_Y(y,y'))},\quad &d_Y(y,y')\leq \pi;\\
r+r', &d_Y(y,y')\geq \pi,\
\end{cases}
\end{equation}
with $z=(r,y)$ and $z'=(r',y').$
Furthermore, about the distance function, we refer the reader to Li \cite[Proposition 1.3, Lemma 3.1]{L2} for the following results.
\begin{lemma}\label{lem:d}There exist constants $c$ and $C$ such that the following property of the distance function holds
$$c(|r-r'|^2+rr' d_Y^2(y,y'))\leq d^2(z,z')\leq C(|r-r'|^2+rr' d_Y^2(y,y'))$$
and
$$ d(z,z')\sim  |r-r'|+\min\{r,r'\} d_Y(y,y').$$ Let $y'\in Y$ and define the ball $B_Y(y',\delta)=\{y\in Y: d_Y(y',y)\leq \delta\}$ and  $z'\in X$ and the ball $B(z',r)=\{z\in X: d(z',z)\leq r\}$. Then there exists $C$ such that
\begin{equation}
\mu_Y(B_Y(y',\delta))\leq C\delta^{n-1},\quad \mu(B(z',r))\leq C r^{n}.
\end{equation}
\end{lemma}

As a consequence, we first have
\begin{lemma}\label{lem:id} For $0<\alpha<1$ and let $z=(r,y)$, there exists $C$ such that
\begin{equation}
\int_{\{z'\in X: r\sim r'\}} d(z,z')^{-(n-\alpha)} d\mu(z')\le Cr^{\alpha}.
\end{equation}
\end{lemma}
\begin{proof} By a direct computation and Lemma \ref{lem:d}, we have 
\begin{align*}
&\int_{r\sim r'} d(z,z')^{-(n-\alpha)} d\mu(z')
\\&\lesssim\int_{r'\sim r}
(\max\{|r-r'|, r' d_Y(y,y')\})^{-(n-\alpha)}r'^{n-1}dr' d\mu_Y(y')\\
&\lesssim\int_{r'\sim r}
\int_{\{y'\in Y: r'd_Y(y,y')\geq|r-r'|\}}(r'd_Y(y,y'))^{-(n-\alpha)} d\mu_Y(y') r'^{n-1}dr'\\&\quad+\int_{r'\sim r}\int_{\{y'\in Y: r'd_Y(y,y')<|r-r'|\}}|r-r'|^{-(n-\alpha)} d\mu_Y(y')r'^{n-1}dr'\\
&\lesssim\int_{r'\sim r}
\sum_{k\geq0}2^{-(n-\alpha)k}|r-r'|^{-(n-\alpha)}\int_{\{y'\in Y: r'd_Y(y,y')\sim 2^k|r-r'|\}} d\mu_Y(y') r'^{n-1}dr'\\&\quad+\int_{r'\sim r}|r-r'|^{-(n-\alpha)}\int_{\{y'\in Y: r'd_Y(y,y')<|r-r'|\}} d\mu_Y(y')r'^{n-1}dr'\\
&\lesssim\int_{r'\sim r}
\sum_{k\geq0}2^{-(n-\alpha)k}|r-r'|^{-(n-\alpha)} \left(\frac{2^k|r-r'|}{r'}\right)^{n-1} r'^{n-1}dr'\\&\quad+\int_{r'\sim r}|r-r'|^{-(n-\alpha)}\left(\frac{|r-r'|}{r'}\right)^{n-1}r'^{n-1}dr'\\
&\lesssim\int_{r'\sim r}
|r-r'|^{-1+\alpha}dr'\lesssim r^\alpha.
\end{align*}
\end{proof}

Next we recall  the Hardy-Littlewood-Sobolev inequality in \cite[Corollary 1.4]{L2}, and
we provide an alternative  argument by using the above lemma.
\begin{proposition}[Hardy-Littlewood-Sobolev]\label{prop:hls}
Let  $0<\alpha<n$, for any  function $f(z)\in L^p(X)$,  let
$$S_{\alpha}f(z)=\int_{X}\frac{f(z')}{d(z,z')^\alpha}\;d\mu(z'),\quad \forall\;z\in X.$$
Then, for any $1<p<q<+\infty$ satisfying
$$1+\frac1q=\frac1p+\frac{\alpha}{n},$$
there exists a constant $A_{p,q}>0$ such that
\begin{equation}\label{equ:equ1}
\big\|S_{\alpha}f\big\|_{L^q}\leq A_{p,q}\|f\|_{L^p},\quad\; f\in L^p(X).
\end{equation}
\end{proposition}

\begin{proof}

From the classical Marcinkiewicz interpolation theorem, we only need to show that there is a constat $C>0$ such that for any $\lambda>0$
\begin{equation}\label{equ:wm}
\mu\big\{z\in X:\quad |S_{\alpha}f(z)|>\lambda\big\}\leq C\Big(\frac{\|f\|_{L^p}}{\lambda}\Big)^q.
\end{equation}
For any $\gamma>0$, define
\begin{equation*}
S_{\alpha}^1f(z)=\int_{d(z,z')\leq\gamma}\frac{f(z')}{d(z,z')^\alpha}\;d\mu(z')
\end{equation*}
and
\begin{equation*}
S_{\alpha}^2f(z)=\int_{d(z,z')>\gamma}\frac{f(z')}{d(z,z')^\alpha}\;d\mu(z').
\end{equation*}
Thus, for any $\tau>0$,
\begin{equation}\label{equ:m2ta}
\begin{split}
&\mu\big\{z\in X: |S_{\alpha}f(z)|>2\tau\big\}\\&\leq
\mu\big\{z\in X: |S_{\alpha}^1f(z)|>\tau\big\}+\mu\big\{z\in X: |S_{\alpha}^2f(z)|>\tau\big\}.
\end{split}
\end{equation}
Without loss of generality, assume that $\|f\|_{L^p}=1$. By H\"older's inequality, we get
\begin{align*}
|S_{\alpha}^2f(z)|\leq&\sum_{k: 2^k\geq \gamma}\frac{1}{2^{k\alpha}}\int_{d(z,z')\sim2^k}|f(z')|\;d\mu(z')\\
\lesssim&\sum_{k: 2^k\geq \gamma}\frac{1}{2^{k\alpha}}\Big(\int_{d(z,z')\sim2^k}\;d\mu(z')\Big)^\frac{1}{p'}\|f\|_{L^p}\\
\lesssim&\sum_{k: 2^k\geq \gamma}\frac{2^\frac{kn}{p'}}{2^{k\alpha}}
\leq C_1\gamma^{-\frac{n}{q}},\quad \frac1{p'}+\frac1q=\frac\alpha n
\end{align*}
where we use that $\mu\big(B\big((s_\ast,m_\ast),r\big)\big)\sim r^n $ in Lemma \ref{lem:d}. Choose $\gamma$ so that $C_1\gamma^{-\frac{n}{q}}=\tau$, then
$$\mu\big\{z: |S_{\alpha}^2f(z)|>\tau\big\}=0.$$
On the other hand, we will show that
\begin{equation}\label{equ:redu}
\|S_\alpha^1 f(z)\|_{L^p}\leq C\gamma^{n-\alpha}.
\end{equation}
Then, we have by Chebyshev's inequality \cite{G}
$$\mu\big\{z: |S_{\alpha}^1f(z)|>\tau\big\}\leq\frac{\|S_\alpha^1 f(X)\|_{L^p}^p}{\tau^p}\leq C\gamma^{p(n-\alpha)}\tau^{-p}=C_2\tau^{-q}.$$
Thus, \eqref{equ:wm} follows  if  let $\lambda=2\tau$.
Now we prove \eqref{equ:redu}.
Recalling $z=(r,y),\;z'=(r',y')$, and using the compactness of $Y$ and the H\"older inequality, we obtain
\begin{align*}
&\|S_\alpha^1 f(z)\|_{L^p}\\
\lesssim&\sum_{k:2^k\leq\gamma}\frac{1}{2^{k\alpha}}\Big(\int_{X}\Big|\int_{d(z,z')\sim2^k}f(z')\;d\mu(z')|^p\;d\mu(z)\Big)^\frac{1}{p}\\
\lesssim&\sum_{k:2^k\leq\gamma}\frac{2^{\frac{kn}{p'}}}{2^{k\alpha}}\Big(\int_{X}\int_{d(z,z')\sim2^k}|f(z')|^p\;d\mu(z')\;d\mu(z)\Big)^\frac{1}{p}\\
\lesssim&\sum_{k:2^k\leq\gamma}\frac{2^{kn}}{2^{k\alpha}}\|f\|_{L^p}
\lesssim\gamma^{n-\alpha}.
\end{align*}

\end{proof}

\vspace{0.2cm}

\subsection{Lorentz spaces} In this subsection, we recall the well-known Lorentz space and some properties of this space for our purpose.
Let $(X,\mu)$ be a $\sigma$-finite measure space and $f: X\to \R$ be a measurable function. Define the distribution function of $f$ as
$$\mu_f(t)=\mu(\{z\in X: |f(z)|> t\}),\quad t>0$$ and its rearrangement function as $$f^*(s)=\inf\{t: \mu_f(t)\leq s\}.$$ For $1\leq p<\infty$ and $1\leq r\leq \infty$, define the Lorentz quasi-norm
\begin{equation*}
\|f\|_{L^{p,r}(X)}=\begin{cases}\Big(\int_0^\infty(s^{\frac1p}f^*(s))^r\frac{ds}{s}\Big)^{1/r}, &\quad 1\leq r<\infty;\\
\sup_{s>0} s^{\frac1p}f^*(s),&\qquad r=\infty.
\end{cases}
\end{equation*}
The Lorentz space  $L^{p,r}(X,\mu)$ denotes the space of complex-valued measurable functions $f$ on X such that its quasi-norm $\|f\|_{L^{p,r}(X)}$ is finite.
From this characterization, $L^{p,\infty}(X)$ is the usual weak $L^p$ space, $L^{p,p}(X)=L^p(X)$ and $L^{p,r}(X)\subset L^{p,\tilde{r}}(X)$ with $r<\tilde{r}$.

We refer to \cite{G} for the following properties of Lorentz space. The first one is the H\"older inequality due to O'Neil \cite{Neil}.
\begin{proposition}[H\"older's inequality in Lorentz space]\label{Lorentz} Let $1\leq p, p_0, p_1<\infty$ and $1\leq r, r_0, r_1\leq \infty$, then
\begin{equation}
\|fg\|_{L^{p,r}}\leq C \|f\|_{L^{p_0,r_0}} \|g\|_{L^{p_1,r_1}}, \quad \frac1p=\frac{1}{p_0}+\frac{1}{p_1}, ~~\frac1r=\frac{1}{r_0}+\frac{1}{r_1}.
\end{equation}
\end{proposition}
The second one is the duality of the Lorentz space.
\begin{proposition}[The dual of Lorentz space]\label{prop:dual} The dual of the Lorentz space  $(L^{p,r}(X))^*=L^{p', r'}(X)$.
\end{proposition}

It is more convenient to use their characterization as real interpolates of Lebesgue spaces. We refer to \cite{BL}.
Suppose that $B_0$ and $B_1$ are two Banach spaces which are continuously embedded into a common topological vector space $\mathcal{V}$,
for $\theta\in (0,1)$ and $r\in [1,\infty]$, the real interpolation space $[B_0, B_1]_{\theta,r}$ consists of the elements $f\in\mathcal{V}$ which
can be written  as $f=\sum\limits_{j\in\Z}f_j$ such that $f_j\in B_0\cap B_1$, $\{2^{-j\theta}\|f_j\|_{B_0}\}_j\in \ell^{r}(\Z)$ and $\{2^{j(1-\theta)}\|f_j\|_{B_1}\}_j\in \ell^{r}(\Z)$. Actually the space
is equipped with the norm
\begin{equation*}
\|f\|_{[B_0, B_1]_{\theta,r}}=\inf_{f=\sum\limits_{j\in\Z}f_j}\left(\left(\sum_{j\in\Z}2^{-jr\theta}\|f_j\|^r_{B_0}\right)^{1/r}+\left(\sum_{j\in\Z}2^{jr(1-\theta)}\|f_j\|^r_{B_1}\right)^{1/r}\right).
\end{equation*}
We have the following from \cite[Theorem 5.3.1]{BL}
\begin{proposition}\label{Lorentz'} Let $1\leq p, p_0, p_1<\infty$ and $1\leq r, r_0, r_1\leq \infty$, then

1) if $p_0\neq p_1$, we have
\begin{equation}
[L^{p_0}, L^{p_1}]_{\theta,r}=[L^{p_0,r_0}, L^{p_1,r_1}]_{\theta,r}=L^{p,r}, \quad \frac1p=\frac{1-\theta}{p_0}+\frac{\theta}{p_1}, ~~1\leq r\leq \infty;
\end{equation}

2) if $p_0=p_1=p$, we have
\begin{equation}
[L^{p,r_0}, L^{p,r_1}]_{\theta,r}=L^{p,r}, \quad \frac1r=\frac{1-\theta}{r_0}+\frac{\theta}{r_1}.
\end{equation}

\end{proposition}

\subsection{The spectral measure} We first use the separation of variable method to analyze the spectral measure $dE_{\sqrt{\LL_V}}$.
In this part, we obtain an explicit expression of half wave operator in terms of series of eigenfunctions which allows us to study a local smoothing estimate but not the dispersive estimate.
Next, in the case $V\equiv0$, we recall an integral expression for a microlocalized spectral measure based on our previous result \cite{ZZ1}. This allows us
to obtain the dispersive estimate for a microlocalized half-wave operator.

Our operator is
\begin{equation}\label{L_V}
\mathcal{L}_V=\Delta_{g}+\frac{V_0(y)}{r^2}.
\end{equation}
Let $\Delta_h$ be the positive Laplace-Beltrami operator on
$(Y,h)$, we suppose that $V_0$ is a smooth function on $Y$ such that
\begin{equation}\label{a1}
\Delta_h+V_0(y)+(n-2)^2/4>0
\end{equation} is strictly positive on $L^2(Y)$ in sense that for any $f\in L^2(Y)\backslash\{0\}$
\begin{equation*}
\left\langle\big(\Delta_h+V_0(y)+(n-2)^2/4\big) f,f\right\rangle_{L^2(Y)}>
0 .
\end{equation*}   Define the set $\chi_\infty$ to be
\begin{equation}\label{set1}
\chi_\infty=\Big\{\nu: \nu=\sqrt{(n-2)^2/4+\lambda};~
\lambda~\text{is eigenvalue
of}~ \Delta_h+V_0(y)\Big\}.
\end{equation}
For $\nu\in\chi_\infty$, let $d(\nu)$ be the multiplicity of
$\lambda_\nu=\nu^2-\frac14(n-2)^2$ as eigenvalue of
$\widetilde{\Delta}_h:=\Delta_h+V_0(y)$. Let $\{\varphi_{\nu,\ell}(y)\}_{1\leq
\ell\leq d(\nu)}$ be the eigenfunctions of $\widetilde{\Delta}_h$, that
is
\begin{equation}\label{eig-v}
\widetilde{\Delta}_h\varphi_{\nu,\ell}=\lambda_{\nu}\varphi_{\nu,\ell},
\quad \langle
\varphi_{\nu,\ell},\varphi_{\nu,\ell'}\rangle_{L^2(Y)}=\delta_{\ell,\ell'}= \begin{cases} 1, \quad \ell=\ell'\\ 0, \quad \ell\neq \ell'.
\end{cases}
\end{equation}
We can decompose $L^2(Y)$ into
\begin{equation*}
L^2(Y)=\bigoplus_{\nu\in\chi_\infty} \mathcal{H}^{\nu}
\end{equation*} where $\mathcal{H}^{\nu}=\text{span}\{\varphi_{\nu,1},\ldots,
\varphi_{\nu,d(\nu)}\}$. Define the orthogonal projection $\pi_{\nu}$ on $f\in L^2(X)$
\begin{equation*}
\pi_{\nu}f=\sum_{\ell=1}^{d(\nu)}\varphi_{\nu,\ell}(y)\int_{Y}f(r,y)
\varphi_{\nu,\ell}(y) dh:= \sum_{\ell=1}^{d(\nu)}\varphi_{\nu,\ell}(y) a_{\nu,\ell}(r).
\end{equation*}
 For any $f\in L^2(X)$, we can write $f$ in the form of separation of variable
\begin{equation}\label{sep.v}
f(z)=\sum_{\nu\in\chi_\infty}\pi_{\nu}f
=\sum_{\nu\in\chi_\infty}\sum_{\ell=1}^{d(\nu)}a_{\nu,\ell}(r)\varphi_{\nu,\ell}(y)
\end{equation}
and furthermore
\begin{equation}\label{norm1}
\|f(z)\|^2_{L^2(Y)}=\sum_{\nu\in\chi_\infty}\sum_{\ell=1}^{d(\nu)}|a_{\nu,\ell}(r)|^2.
\end{equation}
Note that the Riemannian metric $h$ on $Y$ is independent of $r$ ,  we can use the separation of variable method \cite{CT} to write $\mathcal{L}_V$  in the coordinate $(r,y)$ as
\begin{equation}\label{operator-t}
\mathcal{L}_V=-\partial^2_r-\frac{n-1}r\partial_r+\frac1{r^2}\big(\Delta_h+V_0(y)\big).
\end{equation} Let  $\nu>-\frac12$ and $r>0$ and define the Bessel function of order $\nu$ by
\begin{equation}\label{Bessel}
J_{\nu}(r)=\frac{(r/2)^{\nu}}{\Gamma\left(\nu+\frac12\right)\Gamma(1/2)}\int_{-1}^{1}e^{isr}(1-s^2)^{(2\nu-1)/2}\mathrm{d
}s.
\end{equation}
\begin{lemma}\label{lem: J} Let $J_\nu(r)$ be the Bessel function defined in \eqref{Bessel} and $R\gg1$, then there exists a constant $C$ independent of $\nu$ and $R$ such that

\begin{equation}\label{bessel-r}
|J_\nu(r)|\leq
\frac{Cr^\nu}{2^\nu\Gamma(\nu+\frac12)\Gamma(1/2)}\left(1+\frac1{\nu+1/2}\right),
\end{equation}
and
\begin{equation}\label{est:b}
\int_{R}^{2R} |J_\nu(r)|^2 dr \leq C.
\end{equation}
\end{lemma}

\begin{proof} The first one is obtained by a direct computation. The  inequality \eqref{est:b} is  a direct consequence of the asymptotically behavior of Bessel function; see \cite[Lemma 2.2]{MZZ}.

\end{proof}

Let $f\in L^2(X)$, define the Hankel transform of order $\nu$ by
\begin{equation}\label{hankel}
(\mathcal{H}_{\nu}f)(\rho,y)=\int_0^\infty(r\rho)^{-\frac{n-2}2}J_{\nu}(r\rho)f(r,y)r^{n-1}dr.
\end{equation}
On the space $\mathcal{H}^{\nu}$, we see
\begin{equation}\label{2.8}
\begin{split}
\LL_V=A_{\nu}:=-\partial_r^2-\frac{n-1}r\partial_r+\frac{\nu^2-\left(\frac{n-2}2\right)^2}{r^2}.
\end{split}
\end{equation}
Briefly recalling functional calculus on cones \cite{Taylor}, for well-behaved functions $F$, we have by (8.45) in \cite{Taylor}
\begin{equation}\label{funct}
F(\mathcal{L}_V) g(r,y)=\sum_{\nu\in\chi_\infty}\sum_{\ell=1}^{d(\nu)} \varphi_{\nu,\ell}(y) \int_0^\infty F(\rho^2) (r\rho)^{-\frac{n-2}2}J_\nu(r\rho)b_{\nu,\ell}(\rho)\rho^{n-1} d\rho
\end{equation}
where $b_{\nu,\ell}(\rho)=(\mathcal{H}_{\nu}a_{\nu,\ell})(\rho)$ with $g(r,y)=\sum\limits_{\nu\in\chi_\infty}\sum\limits_{\ell=1}^{d(\nu)}a_{\nu,\ell}(r)~\varphi_{\nu,\ell}(y)$.

For $u_0\in L^2(X)$, we write it in the form of separation of variables by \eqref{sep.v}
$$
u_0(z)=\sum_{\nu\in\chi_\infty}\sum_{\ell=1}^{d(\nu)}a_{\nu,\ell}(r)\varphi_{\nu,\ell}(y),
$$
therefore we
can write the half-wave operator by using \eqref{funct} with $F(\rho)=e^{it\rho}$
\begin{equation}\label{s.exp}
\begin{split} e^{it\sqrt{\mathcal{L}_V}}u_0&=\sum_{\nu\in\chi_\infty}\sum_{\ell=1}^{d(\nu)}\varphi_{\nu,\ell}(y)\int_0^\infty(r\rho)^{-\frac{n-2}2}J_{\nu}(r\rho)e^{
it\rho}b_{\nu,\ell}(\rho)\rho^{n-1}d\rho
\\&=\sum_{\nu\in\chi_\infty}\sum_{\ell=1}^{d(\nu)}\varphi_{\nu,\ell}(y)\mathcal{H}_{\nu}\big[e^{
it\rho}b_{\nu,\ell}(\rho)\big](r).
\end{split}
\end{equation}
where $b_{\nu,\ell}(\rho)=(\mathcal{H}_{\nu}a_{\nu,\ell})(\rho)$.  \vspace{0.2cm}

Although we have the expression of the half-wave operator, it is not easy to obtain a dispersive estimate due to
the complicated asymptotic behavior of the Bessel function. In our previous paper \cite{ZZ1}, we derived a microlocalized dispersive estimate for Schr\"odinger from a micro-localized spectral measure expression associated with $\LL_0$.
The result about the micro-localized spectral measure is an analogue of \cite[Proposition 1.5]{HZ} on the asymptotically conic setting.  We record the result on the spectral measure below for convenience.

\begin{proposition}[Proposition 3.1 \cite{ZZ1}]
\label{prop:localized spectral measure} Let $(X,g)$ be metric cone manifold and $\LL_0=\Delta_g$. Then there exists a $\lambda$-dependent  operator partition of unity on
$L^2(X)$
$$
\mathrm{Id}=\sum_{j=0}^{N}Q_j(\lambda),
$$
with $N$ independent of $\lambda$,
such that for each $1 \leq j \leq N$ we can write
\begin{equation}\label{beanQ}\begin{gathered}
(Q_j(\lambda)dE_{\sqrt{\LL_0}}(\lambda)Q_j^*(\lambda))(z,z')=\lambda^{n-1} \Big(  \sum_{\pm} e^{\pm
i\lambda d(z,z')}a_\pm(\lambda,z,z') +  b(\lambda, z, z') \Big),
\end{gathered}\end{equation}
and $0\leq j'\leq N$
\begin{equation}\label{beanQ'}\begin{gathered}
(Q_0(\lambda)dE_{\sqrt{\LL_0}}(\lambda)Q_{j'}^*(\lambda))(z,z')=\lambda^{n-1} c(\lambda, z, z'),
\end{gathered}\end{equation}
with estimates
\begin{equation}\label{bean}\begin{gathered}
\big|\partial_\lambda^\alpha a_\pm(\lambda,z,z') \big|\leq C_\alpha
\lambda^{-\alpha}(1+\lambda d(z,z'))^{-\frac{n-1}2},
\end{gathered}\end{equation}
\begin{equation}\label{beans}\begin{gathered}
\big| \partial_\lambda^\alpha b(\lambda,z,z') \big|\leq C_{\alpha, X}
\lambda^{-\alpha}(1+\lambda d(z,z'))^{-K} \text{ for any } K>0,
\end{gathered}\end{equation}
and
\begin{equation}\label{beanc}\begin{gathered}
\big| \partial_\lambda^\alpha c(\lambda,z,z') \big|\leq C_{\alpha, X}
\lambda^{-\alpha}.
\end{gathered}\end{equation}
Here $d(\cdot, \cdot)$ is the distance on $X$.

\end{proposition}

\vspace{0.2cm}

\subsection{The Littlewood-Paley square function inequality} As a usual reduction to prove Strichartz estimate for wave equation, we may
assume the initial data $u_0$ and $u_1$ are frequency localized in an annulus $\{\lambda\sim 2^k\}$ by using a Littlewood-Paley square function inequality.  To this end, we prove the Littlewood-Paley square function  inequality associated with the positive Laplacian $\LL_0=\Delta_g$ on metric cone.  In \cite{L2},  Li has proved the Gaussian boundedness of heat kernel of $\LL_0$.
One can follow the argument in \cite{BFHM, Z2} to obtain an appropriate
Mikhlin-H\"ormander multiplier theorem  from a spectral multiplier theorem of Alexopolous \cite{Alex} and the heat kernel estimate and then to prove the Littlewood-Paley inequality.
Here we provide an alternative method to show the Littlewood-Paley inequality. The method is based on an estimate on the spectral measure rather than the heat kernel.\vspace{0.2cm}

Now we state the Littlewood-Paley square function estimate.  Let $\varphi\in C_0^\infty(\mathbb{R}\setminus\{0\})$ take values in
$[0,1]$ and be supported in $[1/2,2]$ such that
\begin{equation}\label{dp}
1=\sum_{j\in\Z}\varphi(2^{-j}\lambda),\quad\lambda>0.
\end{equation}

\begin{proposition}\label{prop:square} Let $(X,g)$ be a metric cone of dimensional $n\geq3$ as above, and suppose that $\LL_0=\Delta_g$ is the
Laplace-Beltrami operator on $(X,g)$. Then for $1<p<\infty$,
there exist constants $c_p$ and $C_p$ depending on $p$ such that
\begin{equation}\label{square}
c_p\|f\|_{L^p(X)}\leq
\big\|\big(\sum_{j\in\Z}|\varphi(2^{-j}\sqrt{\LL_0})f|^2\big)^{\frac12}\big\|_{L^p(X)}\leq
C_p\|f\|_{L^p(X)}
\end{equation}
and
\begin{equation}\label{squareL}
c_p\|f\|_{L^{p,2}(X)}\leq
\big\|\big(\sum_{j\in\Z}|\varphi(2^{-j}\sqrt{\LL_0})f|^2\big)^{\frac12}\big\|_{L^{p,2}(X)}\leq
C_p\|f\|_{L^{p,2}(X)}.
\end{equation}
\end{proposition}

\begin{remark}  In this result, we do not consider the influence of the inverse-square potential $V=V_0(y)r^{-2}$.  We remark that the inverse-square type potential plays an important role in the range of $p$ when the potential is negative, for example \cite[Theorem 5.3]{KMVZZ1}.
\end{remark}

\begin{proof}
To prove the Littlewood-Paley square function inequality \eqref{square}, one can follow Stein's \cite{Stein} classical argument
(in $\R^n$) involving Rademacher functions and an appropriate
Mikhlin-H\"ormander multiplier theorem in the following Lemma \ref{lemma:MH}. For more details, we refer the reader to \cite{BFHM, Stein}.
\begin{lemma}\label{lemma:MH}  Let $m\in C^N(\R)$ satisfy the
Mikhlin-type condition for $N\geq\frac n2+1$
\begin{equation}\label{Minhlin}
\sup_{0\leq k\leq
N}\sup_{\lambda\in\R}\Big|\big(\lambda\partial_\lambda\big)^k
m(\lambda)\Big|\leq C<\infty.
\end{equation}
Then for all $1<p<\infty$
\begin{equation}
m(\sqrt{\LL_0}): L^p(X)\to L^p(X)
\end{equation}
is a bounded operator where
$$m(\sqrt{\LL_0})=\int_0^\infty m(\lambda)dE_{\sqrt{L_0}}(\lambda).$$

\end{lemma}

Then the inequality \eqref{squareL} follows from the general Marcinkiewicz interpolation theorem \cite[Theorem 5.3.2]{BL} and dual argument. Indeed, define the
quadratic functional operator for $f\in L^p(X)$
$$\mathcal{G}_{\LL_0}(f)=\left(\sum_{j\in\Z}|\varphi(2^{-j}\sqrt{\LL_0})f|^2\right)^{1/2},$$ then the operator $\mathcal{G}_{\LL_0}$ is sublinear and is bounded
on $L^{1+}$  and $L^{\infty-}$ respectively. Therefore, using the general Marcinkiewicz interpolation theorem \cite[Theorem 5.3.2]{BL}, the operator $\mathcal{G}_{\LL_0}$ is bounded on $L^{p,r}(X)$ for
 all $1<p<\infty$ and $0<r\leq \infty$, hence the case $r=2$ shows the second inequality in  \eqref{squareL}. The other side can be obtained by dual argument.
 \end{proof}

Now our main task here is to show Lemma \ref{lemma:MH}.
\begin{proof}
We adopt the argument which are in spirit of  \cite{GHS2} and \cite{COSY}.
We first prove the spectral measure estimate by using the $TT^*$ argument as given in \cite{GHS2}
\begin{equation}\label{est:spect}
\|dE_{\sqrt{\LL_0}}(\lambda)\|_{L^1(X)\to L^\infty(X)}\leq C\lambda^{n-1}, \quad \lambda>0.
\end{equation}
By Proposition \ref{prop:localized spectral measure}, it is easy to see that
\begin{equation*}
\|Q_j(\lambda)dE_{\sqrt{\LL_0}}(\lambda)Q_j^*(\lambda)\|_{L^1(X)\to L^\infty(X)}\leq C\lambda^{n-1}, \quad \lambda>0.
\end{equation*}
Let $P(\lambda)$ be the Poisson operator associated with $\LL_0$,  then $dE_{\sqrt{\LL_0}}(\lambda)=(2\pi)^{-1}P(\lambda)P(\lambda)^*$ as
shown in \cite{HV}. By using $TT^*$ argument again, it follows that
\begin{equation*}
\|Q_j(\lambda)P(\lambda)\|_{L^2(X)\to L^\infty(X)}\leq C\lambda^{(n-1)/2}, \quad \lambda>0.
\end{equation*}
Note that the partition of unity $Id=\sum_{j=0}^N Q_j$ in Proposition \ref{prop:localized spectral measure}, therefore we have
\begin{equation*}
\|P(\lambda)\|_{L^2(X)\to L^\infty(X)}\leq \sum_{j=0}^N\|Q_j(\lambda)P(\lambda)\|_{L^2(X)\to L^\infty(X)}\leq C\lambda^{(n-1)/2}, \quad \lambda>0.
\end{equation*}
By $TT^*$ argument again, we show \eqref{est:spect}.

From \cite[formula (0.13)]{CT}, it follows  the finite propagation speed of solutions to $(\partial_t^2+\LL_0)u=0$. Hence the operator $\LL_0$ satisfies the finite speed propagation property. By \eqref{est:spect}
and using \cite[Propositions 2.4, 9.1 and Theorems 4.1, 5.1]{COSY}, we have that $m(\sqrt{\LL_0})$ is bounded on $L^p(X)$ for all $1<p<\infty$.

\end{proof}

\vspace{0.2cm}

\section{Sobolev inequality and a generalized Riesz transform} For our purpose, we consider the Sobolev space, Sobolev inequality and
a boundedness of generalized Riesz transform associated with $\LL_V$ in this section. Recall the notation $z=(r,y)$ and $z'=(r',y')$.

For $s\in\R$,  the operator $\LL_V^{\frac{s}2}$ is defined by
\begin{equation}\label{f-op}
\LL_V^{s/2}=\int_0^\infty \lambda^{s} dE_{\sqrt{\LL_V}}(\lambda)
\end{equation}
where $dE_{\sqrt{\LL_V}}$ is the spectral measure of the operator $\sqrt{\LL_V}$.

\begin{definition}[Sobolev space]
For $1\leq p<\infty$ and $s\in\R$, we define
the homogeneous Sobolev space $\dot{H}^s_p(X):=\LL_V^{-\frac s2} L^p(X)$ over $L^p(X)$ which consists of the functions $f$ such that $\LL_V^{s/2}f\in L^p(X)$.
In particular $p=2$, define $\dot{H}^s(X):=\dot{H}^s_2(X)=\LL_V^{-\frac s2} L^2(X)$.
\end{definition}
\begin{remark}For all general $1\leq p<\infty$, due to the influence of the inverse-square potential,
the Sobolev norm defined here is not equivalent to the analogue one defined by the operator $\LL_0$  without the potential.
For example, we refer the reader to \cite{KMVZZ1} for the Euclidean Laplacian with the inverse-square potential. But for $p=2$, the two norms are equivalent.
\end{remark}
The equivalent of the two Sobolev spaces is closely related to a topics about the boundedness of the generalized Riesz transform operator
\begin{equation}
\Delta_g^{\frac s2}\LL_V^{-\frac s2}: L^p(X)\to L^p(X),
\end{equation}
and its reverse operator $\LL_V^{\frac s2}\Delta_g^{-\frac s2}$. In \cite{KMVZZ1}, the authors studied the equivalent norms in which we replace $\LL_V$ by $\Delta+a r^{-2}$ in Euclidean space
by starting from  its heat kernel estimate. However, as far as we know, there is no result about heat kernel estimate of $\LL_V$, even though Li \cite{L2} proved the heat kernel estimate for
$\Delta_g$ on metric cone. Rather than from heat kernel, we study the problem  from the asymptotical behavior of  the resolvent $(\LL_V+1)^{-1}(z,z') $; see \cite[Theorem 4.11, Lemma 5.4]{HL}.

When $0<s<n$, we  can define the operator $\mathcal L_V^{-\frac{s}2}$ by the Riesz potential kernel
$$
\mathcal{L}_V^{-\frac{s}2}(z,z'):=\int_0^\infty \lambda^{1-s}(\LL_V+\lambda^2)^{-1}(z,z') d\lambda.
$$
Before stating the main results of this section, we show the estimates on the kernels
\begin{lemma}\label{lem:kernel} Let $Q(z,z')$ and $G(z,z')$ be the kernels of the operators $\mathcal{L}_V^{-\frac{s}2}$ and $\nabla_g\Delta_g^{\frac{s-1}2}$ respectively.
Assume $0<s<2$, then $Q(z,z')$ satisfies
\begin{equation}\label{equ:Q}
Q(z, z')\lesssim \begin{cases}
r'^{-n}r^{s}(r/r')^{1-\frac n2+\nu_0-s},\quad r<\frac{r'}2;\\
d(z,z')^{-(n-s)},\quad r\sim r';\\
r^{-n+s}(r'/r)^{1-\frac n2+\nu_0},\quad r>2r';
\end{cases}
\end{equation}
and if $0<s\leq 1$ then $G(z,z')$ satisfies
\begin{equation}\label{equ:G}
G(z,z')\lesssim \begin{cases}
r'^{-n}r^{-s}(r/r')^{s-\frac n2+\nu'_0},\quad r'>2r;\\
r^{-1}d(z,z')^{-(n-1+s)},\quad r\sim r';\\
r^{-n-s}(r'/r)^{1-\frac n2+\nu'_0},\quad r'<\frac{r}{2};
\end{cases}
\end{equation}
where $\nu_0$ (resp. $\nu_0'$) is the square root of the smallest eigenvalue of the operator $\Delta_h+V_0(y)+(n-2)^2/4$ (resp. $\Delta_h+(n-2)^2/4$).

\end{lemma}

\begin{remark}Note that $\nu_0'\geq (n-2)/2$, as mentioned in \cite[Remark 4.13]{HL},  one can improve \eqref{equ:G} through replacing $\nu_0'$ by $\nu_1'$,
the square root of the second smallest eigenvalue of  $\Delta_h+(n-2)^2/4$.\end{remark}

\begin{proof} We first estimate $Q(z,z')$. Let $\chi\:[0,\infty)\to [0,1]$ be a smooth cutoff function such that $\chi([0,1/2])=1$ and $\chi([1,\infty))=0$. Define
\begin{align}\label{K12}
Q_1(z,z')=\chi(4r/r')\int_0^\infty \lambda^{1-s}(\LL_V+\lambda^2)^{-1}(z,z') d\lambda;
\\ Q_2(z,z')=\chi(4r'/r)\int_0^\infty \lambda^{1-s}(\LL_V+\lambda^2)^{-1}(z,z') d\lambda;
\end{align}
and
\begin{align}\label{Q0}
Q_0(z,z')=(1-\chi(4r/r')-\chi(4r'/r))\int_0^\infty \lambda^{1-s}(\LL_V+\lambda^2)^{-1}(z,z') d\lambda.
\end{align}
Since $\mathcal{L}_V$ is homogeneous of degree $-2$,  then by scaling we have 
$$(\mathcal{L}_V+\lambda^2)^{-1}(z,z')=\lambda^{n-2}(\mathcal{L}_V+1)^{-1}(\lambda z,\lambda z').$$
Now we consider the boundedness of $Q_1$. By \cite[Theorem 4.11]{HL}, for any $N>0$, we have 
\begin{align}\label{Q1}
|\chi(4r/r')(\mathcal{L}_V+1)^{-1}(z,z')|\lesssim r^{1-\frac n2+\nu_0}r'^{1-\frac n2-\nu_0}\langle r'\rangle^{-N}.
\end{align}
Therefore, for any $N>1-s$ and $s<2$, we have by \eqref{Q1} 
\begin{equation}\label{Q1}
\begin{split}
&\left|Q_1(z,z')\right|\lesssim \left|\int_0^\infty \lambda^{n-1-s}\chi(4r/r')(\LL_V+1)^{-1}(\lambda z, \lambda z') d\lambda\right|\\
&\lesssim r'^{2-n}(r/r')^{1-\frac n2+\nu_0}\left(\int_0^{1/r'} \lambda^{1-s}d\lambda+r'^{-N}\int_{1/r'}^\infty \lambda^{1-s-N}d\lambda\right)\\
&\lesssim r'^{-n}r^{s}(r/r')^{1-\frac n2+\nu_0-s}.
\end{split}
\end{equation}
Similarly we consider the boundedness of $Q_2$. By \cite[Theorem 4.11]{HL} again, we have for any $N>0$
\begin{align}\label{Q2}
|\chi(4r'/r)(\mathcal{L}_V+1)^{-1}(z,z')|\lesssim r'^{1-\frac n2+\nu_0}r^{1-\frac n2-\nu_0}\langle r\rangle^{-N}.
\end{align}
Therefore similarly as estimating \eqref{Q1}, for $s<2$, we have 
\begin{equation}
\begin{split}
\left|Q_2(z,z')\right|\lesssim r^{-n}r^{s}(r'/r)^{1-\frac n2+\nu_0}.
\end{split}
\end{equation}
Finally we estimate $Q_0$. Recall \cite[Lemma 5.4] {HL}, for any $N>0$, we have 
\begin{align*}
\left|(1-\chi(4r/r')-\chi(4r'/r) )(\mathcal{L}_V+1)^{-1}(z,z')\right|\lesssim
\begin{cases}d(z,z')^{2-n}, \quad &d(z,z')\leq 1;\\d(z,z')^{-N}, \quad &d(z,z')\geq 1.
\end{cases}
\end{align*}
Therefore, we compute that by using $d(\lambda z, \lambda z')=\lambda d(z,z')$
\begin{align*}
\left|(1-\chi(4r/r')-\chi(4r'/r) )(\mathcal{L}_V+1)^{-1}(\lambda z,\lambda z')\right|\lesssim
\begin{cases}\lambda^{2-n} d(z,z')^{2-n}, &d(z,z')\leq 1/\lambda;\\\lambda^{-N}d(z,z')^{-N},  &d(z,z')\geq 1/\lambda.
\end{cases}
\end{align*}
We estimate the kernel $Q_0(z,z')$ for $s<2$ and $N>n-s$
\begin{equation*}
\begin{split}
|Q_0(z,z')|&\lesssim \left(d(z,z')^{2-n}  \int_0^{1/d(z,z')} \lambda^{1-s} d\lambda+d(z,z')^{-N}\int_{1/d(z,z')}^\infty \lambda^{n-1-s-N}  d\lambda\right)
\\&\lesssim d(z,z')^{-(n-s)}.
\end{split}
\end{equation*}
We need a modification to prove \eqref{equ:Q} due to the support of $\chi$. For instance, from $Q_1$, we directly see that $Q(z,z')\lesssim r'^{-n}r^{s}(r/r')^{1-\frac n2+\nu_0-s}$ when $r<r'/8$.
On the region $r'/8\leq r\leq r'/2$, since $r'/2\leq |r-r'|\leq d(z,z')\leq r+r'\lesssim r' $ thus $d(z,z')\sim r\sim r'$. Therefore we prove the boundedness of $Q$ on $ r\leq r'/2$.
We also can prove the boundedness on $ r\geq 2r'$ through the same modification argument. Hence we prove \eqref{equ:Q}.\vspace{0.2cm}

We next estimate $G$.
Notice that the derivative $\nabla_g$ is of the form that $r^{-1}$  times a smooth b-derivative
for small $r$, and is a smooth scattering vector field for $r$ large; we refer the reader to \cite{HL} for the b-derivative and scattering vector field. Since $0<s<1$, we can replace the $s$ (resp. $\nu_0$) by $1-s$ (resp. $\nu_0'$)
 to obtain the estimate of  the kernel $\Delta_g^{(s-1)/2}$. Therefore, we finally obtain the estimate of $G$ by multiplying $r^{-1}$, thus we prove \eqref{equ:G}.

\end{proof}

\begin{lemma}\label{lem:tfpq} Let $0\leq s<n,\; 1\leq p,q\leq+\infty$.
Let $K(r,r',y,y')$ be a kernel on the cone $X$. Define the operator
$$Tf=\int_{X} K(z,z')f(z')d\mu(z').$$ If
\begin{equation}
|K(r,r',y,y')|\lesssim \begin{cases} r^{-\alpha}r'^{-\beta},\quad &r\leq r'\\
0,\quad &r>r',
\end{cases}
\end{equation}
and $\alpha+\beta=n-s, \beta>0$, then \begin{equation}\label{equ:kfpq}
\|Tf\|_{L^q(X)}\leq C\|f\|_{L^p(X)}, \quad s=\frac{n}{p}-\frac{n}{q},~
\end{equation}
with
\begin{equation}\label{equ:pqr1}
p\leq q<\frac{n}{\max\{\alpha,0\}}.
\end{equation}

Similarly, if
\begin{equation}
|K(r,r',y,y')|\lesssim \begin{cases}
0,\quad &r<r'\\
 r^{-\alpha}r'^{-\beta},\quad &r\geq r',
\end{cases}
\end{equation}
and $\alpha+\beta=n-s, \alpha>0$, then \eqref{equ:kfpq} holds for
\begin{equation}\label{equ:pqr2}
s=\frac{n}{p}-\frac{n}{q},~q>\frac{n}{\min\{\alpha,n\}}.
\end{equation}
\end{lemma}
\begin{remark} In particular $s=0$, then $q=p$. This special result has been proved in \cite[Corollary 5.9]{HL}. Here, we extend such result to $q\geq p.$
\end{remark}

\begin{proof} We use the argument of  \cite[Corollary 5.9]{HL}. Noting that $d\mu=r^{n-1}dr\;dh$ and the section $Y$ is a compact set, we get
\begin{align}\nonumber
\|Tf\|_{L^q(X)}=&\Big(\int_X\Big|\int_X K(z,z')f(z')\;d\mu(z')\Big|^q\;d\mu(z)\Big)^{1/q}\\\label{equ:k1}
\lesssim&\Big(\int_0^\infty\Big|\int_{r<r'}\tilde{K}(r,r')\tilde{f}(r') r'^{-1}\;dr'\Big|^q\; r^{-1}dr\Big)^{1/q},
\end{align}
where $$ \tilde{f}(r')=r'^{\frac{n}{p}}\int_Y|f(r',y)|\;dh$$
and
$$\tilde{K}(r,r')= r^{-\alpha}r'^{-\beta}r^{\frac{n}q}r'^{\frac{n}{p'}}=(r/r')^{\frac{n}{q}-\alpha}.$$
Perform a substitution $\tilde{r}=\ln r,\; \tilde{r}'=\ln r'$, then
\begin{align}\label{equ:k2}
\|Tf\|_{L^q(X)}
\lesssim&\Big(\int_{\R}\Big|\int_{\tilde{r}<\tilde{r}'}\tilde{K}_1(\tilde{r},\tilde{r}')\tilde{f}(e^{\tilde{r}'})\;d\tilde{r}'\Big|^q\; d\tilde{r}\Big)^{1/q},
\end{align}
with
$$\tilde{K}_1(\tilde{r},\tilde{r}')=e^{(\tilde{r}-\tilde{r}')(\frac{n}{q}-\alpha)}.$$
Then, it is easy to see that
\begin{equation}\label{equ:clamikz}
\sup_{\tilde{r}\in\R}\int_{\tilde{r}<\tilde{r}'}|\tilde{K}_1(\tilde{r},\tilde{r}')|^\sigma\;d\tilde{r}'+\sup_{\tilde{r}'\in\R}\int_{\tilde{r}<\tilde{r}'}|\tilde{K}_1(\tilde{r},\tilde{r}')|^\sigma\;d\tilde{r}<+\infty,
\end{equation}
with  $\big(\tfrac{n}q-\alpha\big)\sigma>0$ guaranteed by \eqref{equ:pqr1}. Especially, taking
$\tfrac1{\sigma}=1+\tfrac1q-\tfrac1p\geq1,$ and using
 generalised Young's inequality, we obtain
\begin{align*}
\|Tf\|_{L^q(X)}\lesssim&\Big(\int_{\R}|\tilde{f}(e^{s'})|^p\;ds'\Big)^{1/p}\\
\lesssim&\Big(\int_0^{+\infty}\int_Y|f(r',y)|^p r'^{n-1}\;dr'\;dh\Big)^{1/p}
\lesssim\|f\|_{L^p(X)}.
\end{align*}
Similarly, we obtain the other case. Hence, Lemma \ref{lem:tfpq} follows.

\end{proof}

We prove the following Sobolev inequality which is well-known in the Euclidean space.
\begin{proposition}[Sobolev inequality for $\LL_V$]\label{P:sobolev}  Let $n\geq 3$ and $\nu_0$ be as above. Suppose $0<s<2$, and $1<p,q<\infty$.  Then
\begin{equation}\label{est:sobolev}
\big\|f(z)\big\|_{L^q(X)}\lesssim \big\|\mathcal L_V^\frac{s}2f\big\|_{L^p(X)}
\end{equation}
holds for $s=\tfrac{n}p-\tfrac{n}q$ and
\begin{equation}\label{est:sobolev_hyp}
\frac{n}{\min\{1+\frac n2+\nu_0-s, n\}}<q<\frac{n}{\max\{\frac n2-1-\nu_0, 0\}}.\end{equation}
\end{proposition}

\begin{proof} The proof follows from Lemma \ref{lem:kernel}.
The estimate \eqref{est:sobolev} is equivalent to
\begin{equation}\label{est:sobolevequi}
\|Tf\|_{L^q(X)}\lesssim\|f\|_{L^p(X)}
\end{equation}
where the operator $T=\mathcal L_V^{-\frac{s}2}$ is defined by the Riesz potential kernel
$$
\mathcal{L}_V^{-\frac{s}2}(z,z'):=\int_0^\infty \lambda^{1-s}(\LL_V+\lambda^2)^{-1}(z,z') d\lambda.
$$ By using Lemma \ref{lem:kernel}, we have for $0<s<2$
\begin{equation*}
\mathcal{L}_V^{-\frac{s}2}(z,z')\lesssim \begin{cases}
r'^{-n}r^{s}(r/r')^{1-\frac n2+\nu_0-s},\quad &r<\frac{r'}2;\\
d(z,z')^{-(n-s)},\quad &r\sim r';\\
r^{-n+s}(r'/r)^{1-\frac n2+\nu_0},\quad &r>2r'.
\end{cases}
\end{equation*}
Then by  using\eqref{est:sobolev_hyp}, we obtain Proposition \ref{P:sobolev} from Lemma \ref{lem:tfpq} when $r<r'/2$ and $r>2r'$ and
from the Hardy-Littlewood-Sobolev inequality in Proposition \ref{prop:hls} when $r\sim r'$.

\end{proof}

\begin{corollary}[Sobolev inequality for $\LL_V$]\label{P:sobolev'}  If $q\geq 2$ and $p\geq 2$ satisfying \eqref{est:sobolev_hyp}, the above result holds for $s>0$.
\end{corollary}
\begin{remark} The restriction on $s$ is $0<s<1+\nu_0$. Indeed, from the facts $p\geq2$ and $2\leq q<n/\max\{\frac n2-1-\nu_0, 0\}$, it follows $s=\frac np-\frac nq<1+\nu_0$.
\end{remark}

\begin{proof}
Choose $\{s_j\}_{j=0}^k$ with $s_0=0, s_k=s$  such that $0<s_{j+1}-s_{j}<1$  and and $\{q_j\}_{j=0}^k$ with $q_0=q, q_k=p$ such that $2\leq q_j<n/\max\{\frac n2-1-\nu_0,0\}$  for $j=0,\cdots k-1$.
Thus, we apply Proposition \ref{prop:hls} to obtain
 \begin{equation*}
\begin{split}
&\|\LL_V^{s_{j}/2} f\|_{L^{q_j}(X)} \lesssim \|\LL_V^{(s_{j+1}-s_j)/2} \LL_V^{s_{j}/2} f\|_{L^{q_{j+1}}(X)}= \|\LL_V^{s_{j+1}/2} f\|_{L^{q_k}(X)}.
\end{split}
\end{equation*}
Therefore, we show
 \begin{equation*}
\begin{split}
\| f(z)\|_{L^{q}(X)}\lesssim \|\LL_V^{s_1/2} f\|_{L^{q_1}(X)}  \lesssim\cdots\lesssim \|\LL_V^{s_k/2} f\|_{L^{q_k}(X)}=\|\LL_V^{s/2} f\|_{L^{p}(X)}.
\end{split}
\end{equation*}
\end{proof}

In the rest of this subsection, we consider the boundedness of the operator
\begin{equation}
\Delta_g^{\frac s2}\LL_V^{-\frac s2}: L^p(X)\to L^p(X), \quad 0<s<1.
\end{equation}
When $s=1$, the boundedness of this operator has been established by Lin-Hassell \cite{HL}. For the following
purpose of the establishment of Strichartz estimate, we need the following result
\begin{proposition}\label{P1} Let $n\geq 4$ and suppose that $s=\frac 2{n-1}$ and $\nu_0$ is in above such that $\nu_0>\frac1{n-1}$. Then the operator
\begin{equation}
\Delta_g^{\frac s2}\LL_V^{-\frac s2}: L^p(X)\to L^p(X), \quad  p=\frac{2(n-1)}{n+1}
\end{equation}
is bounded.
\end{proposition}

Before proving this proposition, we show
\begin{lemma} The following inequality holds for $q\in (1,\infty)$
\begin{equation}\label{rRiesz}
\|\sqrt{\LL_0} f\|_{L^{q}}:=\|\Delta_g^{\frac12}f\|_{L^q}\lesssim \|\nabla_g f\|_{L^q}.\end{equation}
\end{lemma}
\begin{proof}
Indeed partial result is a dual consequence of Riesz transform boundedness. More precisely, Lin-Hassell \cite[Theorem 1.1]{HL}, \cite{L1} has shown
$$\nabla_g\LL_0^{-1/2}: L^{p}(X)\to L^{p}(X)$$ is bounded for $p\in (1, n/\max\{\frac n2-\nu'_1,0\})$ where $\nu'_1>0$ is the square root of the second smallest eigenvalue of the operator
$\Delta_h+(n-2)^2/4$. If $Y=\mathbb{S}^{n-1}$, then $\nu_1'>\frac n2$ since the $k$-th eigenvalue of $\Delta_{\mathbb{S}^{n-1}}$ is $k(k+n-2)$. However, for the general $Y$,
 $\nu'_1>(n-2)/2$, one has the boundedness for $p\in (1,n)$ at least. By the dual argument \cite{AC}, we have
\begin{equation}
\|\sqrt{\LL_0} f\|_{L^{q}}\lesssim \|\nabla_g f\|_{L^{q}}
\end{equation}
for all $q\in (n/(n-1),\infty)$. On the other hand, one can use the method in \cite{GS} to show the following Poincar\'e inequalities for $p=1$
\begin{equation}
\int_{B}|f-f_B|^p d\mu(z') \lesssim r^p\int_{B}|\nabla f|^p d\mu
\end{equation}
where $B=B(z,r)$ and $f_B=\frac1{\mu(B)}\int_B f d\mu$. A result in \cite [Theorem 0.7]{AC} claimed that
the doubling condition and Poincar\'e inequality implies the reverse Riesz transform boundedness. Hence \eqref{rRiesz} holds for $q\in (1,\infty)$
which also was stated in \cite [Page 535]{AC} for our setting.
\end{proof}

Now we prove Proposition \ref{P1}. Write $\Delta_g^{\frac s2}=\Delta_g^{\frac 12}\Delta_g^{\frac {s-1}2}$, by using \eqref{rRiesz}, it suffices to establish
\begin{equation}
\nabla_g\Delta_g^{\frac {s-1}2}\LL_V^{-\frac s2}: L^p(X)\to L^p(X),  \quad s=\frac2{n-1}, \quad p=\frac{2(n-1)}{n+1},~n\geq4.
\end{equation}
Let $G(z,z')$ denote the kernel of the operator $\nabla_g\Delta_g^{\frac {s-1}2}$ and $Q(z',z'')$ denote the kernel of the operator $\LL_V^{-\frac s2}$.
Recall $z=(r,y)$, therefore by using Lemma \ref{lem:kernel}, the kernel $G(z,z')$ and $Q(z',z'')$ satisfy
\begin{equation}\label{equ:gzz}
G(z,z')\lesssim \begin{cases}
r'^{-n}r^{-s}(r/r')^{s-\frac n2+\nu'_0},\quad r'>2r;\\
r^{-1}d(z,z')^{-(n-1+s)},\quad r\sim r';\\
r^{-n-s}(r'/r)^{1-\frac n2+\nu'_0},\quad r'<\frac{r}{2};
\end{cases}
\end{equation}
and \begin{equation}\label{equ:qzz}
Q(z',z'')\lesssim \begin{cases}
r''^{-n}r'^{s}(r'/r'')^{1-\frac n2+\nu_0-s},\quad r'<\frac{r''}2;\\
d(z',z'')^{-(n-s)},\quad r'\sim r'';\\
r'^{-n+s}(r''/r')^{1-\frac n2+\nu_0},\quad r'>2r'';
\end{cases}
\end{equation}
Define the operator $$Tf(z):=\int_X K(z,z'')f(z'')\;d\mu(z''),$$
where the kernel $K(z,z'')$ is given by
$$K(z,z''):=\int_X G(z,z')Q(z',z'')\;d\mu(z').$$
To prove Proposition \ref{P1}, it suffices to show
\begin{proposition}\label{prop:mainest}For $0<s<1$, there exists a constant $C$
\begin{equation}\label{equ:mainest}
\|Tf(z)\|_{L^p(X)}\leq C\|f\|_{L^p(X)},
\end{equation}
provided
\begin{equation}\label{con:p}
\frac{n}{\min\{n, \frac n2+\nu_0+1,\frac n2+1+\nu_0'+s\}}<p<\frac{n}{\max\{0, \frac n2-\nu_0',\frac n2-1-\nu_0+s\}}
\end{equation}
\end{proposition}
We postpone the proof for a moment. Note $\nu_0>1/(n-1)$, $\nu_0'=(n-2)/2$ and $s=2/(n-1)$,   the $p=2(n-1)/(n+1)$ satisfies the condition \eqref{con:p}, hence
it proves Proposition \ref{P1} once we have shown this proposition. \vspace{0.2cm}

{\bf The proof of  Proposition \ref{prop:mainest}.}
We divide the kernel $K(z,z'')$ into several cases. \vspace{0.2cm}

{\bf Case1: $2r\leq\frac{r''}{2}$.}
A simple computation shows
\begin{align*}
K(z,z'')=&\Big(\int_{r'<\frac{r}{2}}+\int_{\frac{r}{2}\leq r'\leq2r}+\int_{r'>2r}\Big)G(z,z')Q(z',z'')\;d\mu(z')\\
=&K_{1,1}(z,z'')+K_{1,2}(z,z'')+K_{1,3}(z,z'').
\end{align*}
{\bf The estimate of $K_{1,1}(z,z'')$:} In this case, since $r'<\frac{r}{2}<\frac{r''}{2}$, we have
$$G(z,z')\lesssim  r^{-n-s}(r'/r)^{1-\frac n2+\nu'_0},$$
and
$$Q(z',z'')\lesssim r''^{-n}r'^{s}(r'/r'')^{1-\frac n2+\nu_0-s}.$$
Hence, we get

\begin{align*}
K_{1,1}(z,z'')\lesssim&\int_{r'<\frac{r}{2}} r^{-n-s}(r'/r)^{1-\frac n2+\nu'_0}r''^{-n}r'^{s}(r'/r'')^{1-\frac n2+\nu_0-s}\;d\mu(z')\\
\lesssim&r^{-\frac{n}{2}-1-s-\nu_0'}r''^{-\frac{n}{2}-1-\nu_0+s}\int_{r'<\frac{r}{2}}
r'^{2-n+\nu_0'+\nu_0}\;d\mu(z')\\
\lesssim&r^{-\frac{n}{2}+1+\nu_0-s}r''^{-\frac{n}{2}-1-\nu_0+s}
\end{align*}
Thus, an application of Lemma \ref{lem:tfpq} yields the $L^p$-boundedness for $K_{1,1}(z,z'')$. \vspace{0.2cm}

{\bf The estimate of $K_{1,2}(z,z'')$:} In this case, since $\frac{r}2<r'<2r<\frac{r''}{2}$, we have
$$G(z,z')\lesssim   r^{-1}d(z,z')^{-(n-1+s)},$$
and
$$Q(z',z'')\lesssim r''^{-n}r'^{s}(r'/r'')^{1-\frac n2+\nu_0-s}.$$

Hence, by Lemma \ref{lem:id}, we obtain 
\begin{align*}
K_{1,2}(z,z'')\lesssim&\int_{\frac{r}{2}\leq r'\leq2r} r^{-1}d(z,z')^{-(n-1+s)}r''^{-n}r'^{s}(r'/r'')^{1-\frac n2+\nu_0-s}\;d\mu(z')\\
\lesssim&r^{-1}r''^{-\frac{n}{2}-1-\nu_0+s}\int_{\frac{r}{2}\leq r'\leq2r}
d(z,z')^{-(n-1+s)}r'^{1-\frac{n}2+\nu_0}\;d\mu(z')\\
\lesssim&r^{-\frac{n}{2}+1+\nu_0-s}r''^{-\frac{n}{2}-1-\nu_0+s}.
\end{align*}
Hence, by using Lemma \ref{lem:tfpq} again, we obtain the $L^p$-boundedness for $K_{1,2}(z,z'')$.  \vspace{0.2cm}

{\bf The estimate of $K_{1,3}(z, z'')$:}  We can further decompose
\begin{align*}
K_{1,3}(z,z'')=&\Big(\int_{2r<r'<\frac{r''}2}+\int_{\frac{r''}2\leq r'\leq2r''}+\int_{r'>2r''}\Big)G(z,z')Q(z',z'')\;d\mu(z')\\
=&K_{1,31}(z,z'')+K_{1,32}(z,z'')+K_{1,33}(z,z'').
\end{align*}

We first consider $K_{1,31}(z,z'')$. In this case, we have $2r<r'<\frac{r''}2.$ Thus,
$$G(z,z')\lesssim r'^{-n}r^{-s}(r/r')^{s-\frac n2+\nu'_0} $$
and
$$Q(z',z'')\lesssim r''^{-n}r'^{s}(r'/r'')^{1-\frac n2+\nu_0-s}.$$
This implies
\begin{align*}
K_{1,31}(z,z'')\lesssim&\int_{2r<r'<\frac{r''}2}r'^{-n}r^{-s}(r/r')^{s-\frac n2+\nu'_0} r''^{-n}r'^{s}(r'/r'')^{1-\frac n2+\nu_0-s}\;d\mu(z')\\
\lesssim&r^{-\frac{n}2+\nu_0'}r''^{-\frac{n}2-1-\nu_0+s}\int_{2r<r'<\frac{r''}2}r'^{1-n-s-\nu_0'+\nu_0}\;d\mu(z')\\
\lesssim&\begin{cases}r^{-\frac n2+\nu_0'}r''^{-\frac n2-\nu_0'},\quad -s+\nu_0-\nu_0'+1>0;\\
r^{-\frac{n}2+\nu_0+1-s}r''^{-\frac{n}2-1-\nu_0+s}\quad -s+\nu_0-\nu_0'+1<0.
\end{cases}
\end{align*}
Thus, an application of Lemma \ref{lem:tfpq} yields the $L^p$-boundedness for $K_{1,31}(z,z'')$.

Next consider $K_{1,32}(z,z'')$. In this term, we have $2r<\frac{r''}2<r'<2r''.$ Thus,
$$G(z,z')\lesssim r'^{-n}r^{-s}(r/r')^{s-\frac n2+\nu'_0}$$
and
$$Q(z',z'')\lesssim d(z',z'')^{-(n-s)}.$$
This implies
\begin{align*}
K_{1,32}(z,z'')\lesssim&\int_{\frac{r''}2<r'<2r''}r'^{-n}r^{-s}(r/r')^{s-\frac n2+\nu'_0} d(z',z'')^{-(n-s)}\;d\mu(z')\\
\lesssim&r^{-\frac{n}2+\nu_0'}\int_{\frac{r''}2<r'<2r''}r'^{-\frac n2-\nu'_0-s} d(z',z'')^{-(n-s)}\;d\mu(z')\\
\lesssim&r^{-\frac{n}2+\nu_0'}r''^{-\frac{n}2-\nu_0'}.
\end{align*}
Finally, we consider $K_{1,33}(z,z'')$.  In this case, we have $2r<2r''<r'.$ Thus,
$$G(z,z')\lesssim r'^{-n}r^{-s}(r/r')^{s-\frac n2+\nu'_0}$$
and
$$Q(z',z'')\lesssim r'^{-n+s}(r''/r')^{1-\frac n2+\nu_0}.$$
Since $\nu_0'>(n-2)/2$, this implies
\begin{align*}
K_{1,33}(z,z'')\lesssim&\int_{r'>2r''}r'^{-n}r^{-s}(r/r')^{s-\frac n2+\nu'_0} r'^{-n+s}(r''/r')^{1-\frac n2+\nu_0}\;d\mu(z')\\
\lesssim&r^{-\frac{n}2+\nu_0'} r''^{1-\frac{n}2+\nu_0}\int_{r'>2r''}r'^{-(n+\nu'_0+\nu_0+1)}\;d\mu(z')\\
\lesssim&r^{-\frac{n}2+\nu_0'}r''^{-\frac{n}2-\nu_0'}.
\end{align*}
Therefore, by using Lemma \ref{lem:tfpq}, we obtain the boundedness of $K_{1,3}$. In sum,  in the case $2r<\frac{r''}2$, we prove $K(z,z'')$ is bounded as an operator on $L^p(X)$ provided
\begin{equation}
p<\frac{n}{\max\{0, \frac n2-\nu_0',\frac n2-1-\nu_0+s\}}.
\end{equation}

{\bf Case 2: $\frac{r}2>2r''.$} We decompose
\begin{align*}
K(z,z'')=&\Big(\int_{r'<\frac{r''}{2}}+\int_{\frac{r''}{2}\leq r'\leq2r''}+\int_{r'>2r''}\Big)G(z,z')Q(z',z'')\;d\mu(z')\\
=&K_{2,1}(z,z'')+K_{2,2}(z,z'')+K_{2,3}(z,z'').
\end{align*}

{\bf The estimate of $K_{2,1}$:} In this region, we have $r'<\frac{r''}2<\frac{r}2.$ And so
$$G(z,z')\lesssim r^{-n-s}(r'/r)^{1-\frac n2+\nu'_0}$$
and
$$Q(z',z'')\lesssim r''^{-n}r'^{s}(r'/r'')^{1-\frac n2+\nu_0-s}.$$
Hence, we get
\begin{align*}
K_{2,1}(z,z'')\lesssim&\int_{r'<\frac{r''}{2}}r^{-n-s}(r'/r)^{1-\frac n2+\nu'_0}r''^{-n}r'^{s}(r'/r'')^{1-\frac n2+\nu_0-s}\;d\mu(z')\\
\lesssim&r^{-\frac{n}{2}-s-1-\nu_0'}r''^{-\frac{n}{2}-1-\nu_0+s}\int_{r'<\frac{r''}{2}}
r'^{2-n+\nu_0'+\nu_0}\;d\mu(z')\\
\lesssim&r^{-\frac{n}{2}-s-1-\nu_0'}r''^{-\frac{n}{2}+s+1+\nu_0'}.
\end{align*}

{\bf The estimate of $K_{2,2}$:} In this region, we have $\frac{r''}2<r'<2r''<\frac{r}2.$ And so
$$G(z,z')\lesssim r^{-n-s}(r'/r)^{1-\frac n2+\nu'_0}$$
and
$$Q(z',z'')\lesssim d(z',z'')^{-(n-s)}.$$
Hence, we get
\begin{align*}
K_{2,1}(z,z'')\lesssim&\int_{\frac{r''}{2}<r'<2r''} r^{-n-s}(r'/r)^{1-\frac n2+\nu'_0}d(z',z'')^{-(n-s)}\;d\mu(z')\\
\lesssim&r^{-\frac{n}{2}-1-s-\nu_0'}\int_{\frac{r''}{2}<r'<2r''}
r'^{1-\frac n2+\nu'_0}d(z',z'')^{-(n-s)}\;d\mu(z')\\
\lesssim&r^{-\frac{n}{2}-1-\nu_0'-s}r''^{-\frac{n}{2}+1+\nu_0'+s}.
\end{align*}

{\bf The estimate of $K_{2,3}$:} We further decompose
\begin{align*}
K_{2,3}(z,z'')=&\Big(\int_{2r''<r'<\frac{r}2}+\int_{\frac{r}2\leq r'\leq2r}+\int_{r'>2r}\Big)G(z,z')Q(z',z'')\;d\mu(z')\\
=&K_{2,31}(z,z'')+K_{2,32}(z,z'')+K_{2,33}(z,z'').
\end{align*}

{\bf The contribution of $K_{2,31}$:}
 In this region, we have $2r''<r'<\frac{r}2.$ And so
$$G(z,z')\lesssim r^{-n-s}(r'/r)^{1-\frac n2+\nu'_0}$$
and
$$Q(z',z'')\lesssim r'^{-n+s}(r''/r')^{1-\frac n2+\nu_0}.$$
Hence, we get
\begin{align*}
K_{2,31}(z,z'')\lesssim&\int_{2r''<r'<\frac{r}2} r^{-n-s}(r'/r)^{1-\frac n2+\nu'_0} r'^{-n+s}(r''/r')^{1-\frac n2+\nu_0}\;d\mu(z')\\
\lesssim&r^{-\frac{n}{2}-1-s-\nu_0'}r''^{1-\frac{n}2+\nu_0}\int_{2r''<r'<\frac{r}2}
r'^{-n+\nu'_0+s-\nu_0}\;d\mu(z')\\\lesssim&\begin{cases}r^{-\frac n2-\nu_0-1}r''^{1-\frac n2+\nu_0},\quad s-\nu_0+\nu_0'-1>0;\\
r^{-\frac{n}2-\nu_0'-1-s}r''^{-\frac{n}2+1+s+\nu_0'}\quad s-\nu_0+\nu_0'-1<0.
\end{cases}
\end{align*}

{\bf The contribution of $K_{2,32}$:}
 In this region, we have $2r''<\frac r2<r'<2r.$ And so
$$G(z,z')\lesssim r^{-1}d(z,z')^{-(n-1+s)}$$
and
$$Q(z',z'')\lesssim r'^{-n+s}(r''/r')^{1-\frac n2+\nu_0}.$$
Hence, we get
\begin{align*}
K_{2,32}(z,z'')\lesssim&\int_{\frac r2<r'<2r} r^{-1}d(z,z')^{-(n-1+s)} r'^{-n+s}(r''/r')^{1-\frac n2+\nu_0}\;d\mu(z')\\
\lesssim&r^{-\frac{n}{2}-2+s-\nu_0}r''^{1-\frac{n}2+\nu_0}\int_{\frac r2<r'<2r}
d(z,z')^{-(n-1+s)}\;d\mu(z')\\\lesssim&r^{-\frac{n}{2}-1-\nu_0}r''^{1-\frac{n}2+\nu_0}.\end{align*}

{\bf The contribution of $K_{2,33}$:}
 In this region, we have $2r''<2r<r'.$ And so
$$G(z,z')\lesssim r'^{-n}r^{-s}(r/r')^{s-\frac n2+\nu'_0}$$
and
$$Q(z',z'')\lesssim r'^{-n+s}(r''/r')^{1-\frac n2+\nu_0}.$$
Hence, we get
\begin{align*}
K_{2,33}(z,z'')\lesssim&\int_{2r<r'} r'^{-n}r^{-s}(r/r')^{s-\frac n2+\nu'_0} r'^{-n+s}(r''/r')^{1-\frac n2+\nu_0}\;d\mu(z')\\
\lesssim&r^{-\frac{n}{2}+\nu_0'}r''^{1-\frac{n}2+\nu_0}\int_{2r<r'}
r'^{-n-\nu'_0-\nu_0-1}\;d\mu(z')\\\lesssim&r^{-\frac{n}{2}-1-\nu_0}r''^{1-\frac{n}2+\nu_0}.\end{align*}
Overall, in the case $\frac{r}2>2r$, by using Lemma \ref{lem:tfpq}, we show $K(z,z'')$ is bounded as an operator on $L^p(X)$ provided
\begin{equation}
p>\frac{n}{\min\{n, \frac n2+\nu_0+1,\frac n2+1+\nu_0'+s\}}.
\end{equation}

{\bf Case 3: $\frac{r''}4\leq r\leq4r''$.} We decompose
\begin{align*}
K(z,z'')=&\Big(\int_{r'<\frac{r}{2}}+\int_{\frac{r}{2}\leq r'\leq2r}+\int_{r'>2r}\Big)G(z,z')Q(z',z'')\;d\mu(z')\\
=&K_{3,1}(z,z'')+K_{3,2}(z,z'')+K_{3,3}(z,z'').
\end{align*}

{\bf The estimate of $K_{3,1}$:} In this region, we have $r'<\frac{r}2\leq 2r''.$ If $r'\geq \frac{r''}2$, then one has $r\sim r'\sim r''$ which can be done as treating $K_{3,2}$.
 Hence we only consider $r'<\frac{r''}2$, and so
$$G(z,z')\lesssim r^{-n-s}(r'/r)^{1-\frac n2+\nu'_0}$$
and
$$Q(z',z'')\lesssim r''^{-n}r'^{s}(r'/r'')^{1-\frac n2+\nu_0-s}.$$
Hence, we get
\begin{align*}
K_{3,1}(z,z'')\lesssim&\int_{r'<\frac{r''}{2}} r^{-n-s}(r'/r)^{1-\frac n2+\nu'_0}r''^{-n}r'^{s}(r'/r'')^{1-\frac n2+\nu_0-s}\;d\mu(z')\\
\lesssim&r^{-\frac{n}{2}-1-s-\nu_0'}r''^{-\frac{n}{2}-1-\nu_0+s}\int_{r'<\frac{r''}{2}}
r'^{2-n+\nu_0'+\nu_0}\;d\mu(z')\\
\lesssim&r^{-\frac{n}{2}-1-\nu_0'-s}r''^{-\frac{n}{2}+1+\nu_0'+s}.
\end{align*}
When $r\sim r''$, it is easy to prove that
\begin{align*}
\int_{r\sim r''}K_{3,1}(z,z'') d\mu(z'')\lesssim1;   \quad \int_{r\sim r''}K_{3,1}(z,z'') d\mu(z)\lesssim1.
\end{align*}

{\bf The estimate of $K_{3,2}$:} In this region, we have $r'\sim r\sim r''.$ And so
$$G(z,z')\lesssim r^{-1}d(z,z')^{-(n-1+s)}$$
and
$$Q(z',z'')\lesssim d(z',z'')^{-(n-s)}.$$
Therefore, we prove
\begin{align*}
\int_{r\sim r''}K_{3,2}(z,z'') d\mu(z'')\lesssim&\int_{r\sim r''} \int_{r'\sim r} r^{-1}d(z,z')^{-(n-1+s)}d(z',z'')^{-(n-s)}\;d\mu(z')d\mu(z'')\\
\lesssim&r^{-1}\int_{r'\sim r} d(z,z')^{-(n-1+s)}\int_{r'\sim r''} d(z',z'')^{-(n-s)}d\mu(z'')\;d\mu(z')\\
\lesssim&1.
\end{align*}
Similarly, we can prove $\int_{r\sim r''}K_{3,2}(z,z'') d\mu(z)\lesssim 1$.

{\bf The estimate of $K_{3,3}$:} In this region, we have $r'>2r\geq \frac {r''}2.$ Similarly, we only consider $r'>2r''$. And so
$$G(z,z')\lesssim r'^{-n}r^{-s}(r/r')^{s-\frac n2+\nu'_0}$$
and
$$Q(z',z'')\lesssim r'^{-n+s}(r''/r')^{1-\frac n2+\nu_0}.$$
Hence, we get
\begin{align*}
K_{3,3}(z,z'')\lesssim&\int_{r'>2r}r'^{-n}r^{-s}(r/r')^{s-\frac n2+\nu'_0} r'^{-n+s}(r''/r')^{1-\frac n2+\nu_0}\;d\mu(z')\\
\lesssim&r^{-\frac{n}{2}+\nu_0'}r''^{1-\frac{n}{2}+\nu_0}\int_{r'>2r}
r'^{-n-1-\nu_0'-\nu_0}\;d\mu(z')\\
\lesssim&r^{-\frac{n}{2}-1-\nu_0}r''^{1-\frac{n}{2}+\nu_0}.
\end{align*}
Note that $r\sim r''$, it is easy to prove that
\begin{align*}
\int_{r\sim r''}K_{3,1}(z,z'') d\mu(z'')\lesssim1;   \quad \int_{r\sim r''}K_{3,1}(z,z'') d\mu(z)\lesssim1.
\end{align*}
To conclude, in the case that $r\sim r''$, by using Schur test lemma, we prove $K(z,z'')$ is bounded on $L^p(X)$ for all $1<p<\infty$.
Collecting all the cases, therefore we finish the proof of Proposition \ref{prop:mainest}.

\vspace{0.2cm}

\section{Strichartz estimates for wave equation with $\LL_0$}

In this section, we prove the Strichartz estimates for wave equation associated with $\LL_0$, i.e. without potential, that is, the result (i) of Theorem \ref{thm:Strichartz} when $V=0$.
The argument here is close to  \cite{HZ,Z2} but with necessary modifications.
For the sake of being self-contained and convenient, we sketch the main steps.
\subsection{Microlocalized propagator}

We begin to decompose the half-wave propagator by using the partition of unity $1=\sum\limits_{k\in\Z}\varphi(2^{-k}\lambda)$ as in \eqref{dp}.
Define
\begin{equation}\label{Uk}
\begin{split}
U_k(t) &= \int_0^\infty e^{it\lambda}
\varphi(2^{-k}\lambda) dE_{\sqrt{\LL_0}}(\lambda),\quad k\in\Z.
\end{split}
\end{equation}
We further microlocalize (in phase space) the half-wave propagators adapting to the partition of
unity operator
\begin{equation}\label{Uitj}
\begin{split}
U_{j,k}(t) = \int_0^\infty e^{it\lambda} \varphi(2^{-k}\lambda) Q_j(\lambda)
dE_{\sqrt{\LL_0}}(\lambda), \quad 0 \leq j \leq N,
\end{split}\end{equation}
where $Q_j(\lambda)$ is as in Proposition \ref{prop:localized spectral measure}.
Then the operator $U_{j,k}(t) U_{j,k}(s)^*$
is given
\begin{equation}
U_{j,k}(t) U_{j,k}(s)^* =  \int e^{i(t-s)\lambda} \varphi(2^{-k}\lambda)   Q_j(\lambda)
dE_{\sqrt{\LL_0}}(\lambda) Q_j(\lambda)^*.
\label{Uiti2}\end{equation}

\subsection{$L^2$-estimate and dispersive estimate}
In this subsection, we prove the two key estimates, i.e. the energy estimate and dispersive estimate. Before stating our result, we recall two results in \cite{HZ}.
The results can be  directly applied to our setting if we consider the problems on the region away from the cone tip, in which as mentioned in the introduction they almost are the same.
Recall that $Q_j$ with $j\geq1$ are micro-localized away from the cone tip.\vspace{0.2cm}

By using \cite[Lemma 8.2]{HZ} (see also \cite[Lemmas 5.3 and 5.4]{GH}), we can divide $(j,j')$, $1 \leq j,j' \leq N$ into three classes
$$
\{ 1, \dots, N \}^2 = J_{near} \cup J_{not-out} \cup J_{not-inc},
$$
so that
\begin{itemize}
\item if $(j,j') \in J_{near}$, then $Q_j(\lambda) \specl Q_{j'}(\lambda)^*$ satisfies the conclusions of Proposition~\ref{prop:localized spectral measure};

\item if $(j,j') \in J_{non-inc}$, then $Q_j(\lambda)$ is not incoming-related to $Q_{j'}(\lambda)$ in the sense that no point in the operator wavefront set (microlocal support)
of $Q_j(\lambda)$ is related to a point in the operator wavefront
set of $Q_{j'}(\lambda)$ by backward bicharacteristic flow;

\item if $(j,j') \in J_{non-out}$, then $Q_j(\lambda)$ is not outgoing-related to $Q_{j'}(\lambda)$ in the sense that no point in the operator wavefront set  of
$Q_j(\lambda)$ is related to a point in the operator wavefront set
of $Q_{j'}(\lambda)$ by forward bicharacteristic flow.
\end{itemize}

And we further exploit the not-incoming or not-outgoing property of
$Q_j(\lambda)$ with respect to $Q_{j'}(\lambda)$ to obtain the Schwartz kernel of  $Q_j(\lambda) \specl Q_{j'}(\lambda)^*$

\begin{lemma}\label{lem:sign}For $\lambda> 0$ and $(j,j') \in J_{non-out}$.  Then,
we can write the Schwartz kernel of $Q_j(\lambda) \specl Q_{j'}(\lambda)^*$ as a multiple of $|dg dg'|^{1/2} |d\lambda|$
as the sum of a finite number of terms of the form
\begin{gather}
 \int_{\R^k} e^{i\lambda r\Phi(y,y',\sigma,r,v)}\lambda^{n-1+k/2} r^{-(n-1)/2+k/2}a(\lambda,y,y',\sigma,r,v)dv \quad \mathrm{ or } \label{QiEQj-hi}\\
  \int_{\R^{k-1}} \int_0^\infty e^{i\lambda r\Phi(y,y',\sigma,r,v,s)}\lambda^{n-1+k/2} \big(\frac{1}{rs}\big)^{(n-1)/2 - k/2}
  s^{n-2} a(\lambda,y,y',\sigma,r,v,s) \, ds \, dv  \label{QiEQj-s}
  \end{gather}
in the region $\sigma = r'/r \leq 2$, $r \geq \delta$, or
\begin{gather}
  \int_{\R^k} e^{i\lambda \Phi(z, z', v)} \lambda^{n-1+k/2} a(\lambda, z, z', v) \, dv \label{QiEQj}
\end{gather}
in the region $r \leq \delta, r' \leq \delta$,
where in each case, $\Phi < - \epsilon < 0$ and $a$ is a smooth function compactly supported in the $v$ and $s$ variables (where present), such that $|(\lambda\partial_\lambda)^N a|\leq C_N$.
In each case, we may assume that $k \leq n-1$; if $k=0$ in \eqref{QiEQj-hi} or \eqref{QiEQj}, or $k=1$ in \eqref{QiEQj-s} then there is no variable $v$, and no $v$-integral. Again, the key point is that in each expression, the phase function is strictly negative.

If, instead, $Q_j$ is not incoming-related to $Q_{j'}$, then the same
conclusion holds with the reversed sign: the Schwartz kernel can be
written as a finite sum of terms with a strictly positive phase
function.
\end{lemma}

\begin{remark}
For $\sigma \geq 1/2$, the Schwartz kernel has a similar description, as follows immediately from the symmetry of the kernel under interchanging the left and right variables.
\end{remark}

\begin{proof} Note $(j,j')\in J_{non-out}$, thus $j,j'\geq1$.
Since $j,j'\geq1$ away from cone tip, this result is essentially proved in \cite[Lemma 8.3, Lemma 8.5]{HZ}.
Since our setting has scaling symmetry, we do not need to state the result in high and low frequency respectively. The key point is that the sign of the phase function
can be determined.
\end{proof}

The main results of this subsection are
the $L^2$-estimate and dispersive estimates.
\begin{proposition}\label{prop:Dispersive} Let $U_{j,k}(t)$ be defined in \eqref{Uitj}.
Then there exists a constant $C$ independent of $t, z, z'$ for all
$j,j'\geq 0, k\in\Z$ such that
\begin{equation}\label{energy}
\|U_{j,k}(t)\|_{L^2\rightarrow L^2}\leq C,\end{equation}
and the following dispersive estimates on $U_{j,k}(t) U_{j',k}(s)^*$ hold:

\begin{itemize}
\item If $(j,j') \in J_{near}$ or $(j,j')=(0,j'), (j,0)$, then for all $t \neq s$ we have
\begin{equation}
\big\|U_{j,k}(t)U^*_{j',k}(s)\big\|_{L^1\rightarrow L^\infty}\leq C
2^{k(n+1)/2}(2^{-k}+|t-s|)^{-(n-1)/2}, \label{UiUjnear}\end{equation}

\item If $(j,j')$ such that $Q_j$ is not outgoing related to
$Q_{j'}$, and $t<s$, then
\begin{equation}
\big\|U_{j,k}(t)U^*_{j',k}(s)\big\|_{L^1\rightarrow L^\infty}\leq C
2^{k(n+1)/2}(2^{-k}+|t-s|)^{-(n-1)/2}, \label{UiUj}\end{equation}

\item Similarly, if $(j,j')$ such that $Q_j$ is not incoming
related to $Q_{j'}$, and $s<t$, then
\begin{equation}
\big\|U_{j,k}(t)U^*_{j',k}(s)\big\|_{L^1\rightarrow L^\infty}\leq C
2^{k(n+1)/2}(2^{-k}+|t-s|)^{-(n-1)/2}. \label{UiUj2}\end{equation}
\end{itemize}

\end{proposition}

\begin{remark} The dispersive inequalities \eqref{UiUj} and \eqref{UiUj2} are used to prove endpoint to endpoint inhomogeneous Strichartz estimate; see Section \ref{subsec:4.2}.
\end{remark}

\begin{proof}
The inequalities \eqref{energy} and \eqref{UiUjnear} are essentially proved \cite[Section 3]{Z2}. Indeed,
note that the operators $\varphi(2^{-k}\lambda)$ and $Q_j(\lambda)$ are bounded on $L^2$, thus the microlocalized propagator $U_{j,k}(t)$ is bounded from $L^2(X)$ to itself due to the spectral theory on Hilbert space.
From above result, if $(j,j') \in J_{near}$ or $(j,j')=(0,j'), (j,0)$, we have the expression of microlocalized spectral mearsue in Proposition \ref{prop:localized spectral measure} which is same as the one used in \cite{Z2}.
Then by the stationary phase argument, we have
 \begin{equation}\label{dispersive}
\begin{split}
\Big|\int_0^\infty e^{it\lambda}\varphi(2^{-k}\lambda) \big(Q_j(\lambda)
&dE_{\sqrt{\LL_0}}(\lambda)Q_{j'}^*(\lambda)\big)(z,z')
d\lambda\Big|\\&\leq C 2^{k(n+1)/2}(2^{-k}+|t|)^{-(n-1)/2}
\end{split}
\end{equation}
where  $\varphi\in C_c^\infty([1/2, 2])$ and takes value in
$[0,1]$. We refer the reader to \cite[Section 3]{Z2} for details.

We only prove \eqref{UiUj2} since the argument to prove \eqref{UiUj} is analogous. Assume that $Q_j$
is not incoming-related to $Q_{j'}$, and then consider \eqref{UiUj2}.
 By \cite[Lemma 5.3]{HZ}, $U_{j,k}(t) U_{j',k}(s)^*$ is given by
\begin{equation}
\int_0^\infty e^{i(t-s)\lambda} \tilde{\varphi}(2^{-k}\lambda)\big(Q_j(\lambda)
dE_{\sqrt{\mathbf{\LL_0}}}(\lambda)Q^*_{j'}(\lambda)\big)(z,z')d\lambda,\quad \tilde{\varphi}=\varphi^2.
\label{UiUjint}\end{equation} Then we need to show that for $s<t$ and $k\in\Z$
\begin{equation*}
\begin{split}
&\Big|\int_0^\infty e^{i(t-s)\lambda}\tilde{\varphi}(2^{-k}\lambda) \big(Q_j(\lambda)
dE_{\sqrt{\LL_0}}(\lambda)Q^*_{j'}(\lambda)\big)(z,z')
d\lambda\Big|\\&\leq C 2^{k(n+1)/2}(2^{-k}+|t-s|)^{-(n-1)/2}.
\end{split}
\end{equation*}
By scaling, it suffices to show $k=0$, that is
\begin{equation}\label{k=0}
\Big|\int_0^\infty e^{i(t-s)\lambda} \tilde{\varphi}(\lambda) \big(Q_j(\lambda)
dE_{\sqrt{\LL_0}}(\lambda)Q^*_{j'}(\lambda)\big)(z,z')
d\lambda\Big|\leq C (1+|t-s|)^{-(n-1)/2}.
\end{equation}
If $t-s<1$, since $\tilde{\varphi}$ is compactly supported, the estimate follows from the uniform boundedness of \eqref{QiEQj-hi}-\eqref{QiEQj}.
Now we consider $t-s \geq 1$. Let $\phi\in C_c^\infty([\frac12,2])$ be  such that
$\sum_m\phi(2^{-m}(t-s)\lambda)=1$, define
$$\phi_0((t-s)\lambda)=\sum_{m\leq0}\phi(2^{-m}(t-s)\lambda).$$  Plug the decomposition
$$
1 = \phi_0((t-s)\lambda) + \sum_{m \geq 1} \phi_m((t-s)\lambda), \quad \phi_m(\lambda) :=  \phi(2^{-m}\lambda)
$$
 into the integral  \eqref{k=0}. In addition, we substitute for  $Q_j(\lambda)
dE_{\sqrt{\LL_0}}(\lambda)Q^*_{j'}(\lambda)$ one of the
expressions in Lemma~\ref{lem:sign} to obtain
\begin{equation}
\begin{split}
&\Big|\int_0^\infty e^{i(t-s)\lambda} \tilde{\varphi}(\lambda) \phi_0((t-s)\lambda) \big(Q_j(\lambda)
dE_{\sqrt{\LL_0}}(\lambda)Q^*_{j'}(\lambda)\big)(z,z')
d\lambda\Big|\\&\leq \int_0^\infty \lambda^{n-1} \tilde{\varphi}(\lambda)\phi_0((t-s)\lambda) d\lambda\leq C |t-s|^{-n}.
\end{split}
\end{equation}
Hence it implies \eqref{k=0} since $|t-s|>1$.

For $m \geq 1$, we substitute again one of the expressions in Lemma~ \ref{lem:sign}. Since the other cases follow from the similar argument,
we only consider the expression  \eqref{QiEQj}.   Define $\lambdabar = (t-s)\lambda$, we obtain by scaling
\begin{equation*}\begin{gathered}
\int_0^\infty \int_{\R^{k}}e^{i(t-s)\lambda}
e^{i\lambda\Phi(z,z',v)}\lambda^{n-1+k/2}\tilde{\varphi}(\lambda)a(\lambda,z,z',v)\phi_m(
(t-s)\lambda) \, dv \, d\lambda
\\
= (t-s)^{-n - \frac{k}{2}}\int_0^\infty \int_{\R^{k}}
e^{i\big(\lambdabar+\frac{\lambdabar\Phi(z,z',v)}{t-s}\big)}\lambdabar^{n-1+k/2}
\tilde{\varphi}(\frac{\lambdabar}{t-s})a(\frac{\lambdabar}{t-s},y,y',\sigma,v) \phi_m(\lambdabar) \,
dv \, d\lambdabar. \end{gathered} \label{intt}
\end{equation*}
We observe that the overall exponential factor is invariant under the differential operator
$$
L = \frac{-i}{1+ \Phi/(t-s)}
\frac{\partial}{\partial \lambdabar}.
$$
Note that its adjoint is $L^t = -L$,
we apply $L^N$  to the exponential factors, and integrate by
parts $N$ times. Since $\Phi \geq 0$ according to Lemma~\ref{lem:sign}, and
since we have an estimate $|(\lambdabar \partial_{\lambdabar} )^N (\tilde{\varphi} a)| \leq
C_N$,  we gain a factor $\lambdabar^{-1}
\sim 2^{-m}$ each time, thus we estimate for $t-s>1$
\begin{equation*}\begin{split}
&(t-s)^{-n - \frac{k}{2}}\int_0^\infty \int_{\R^{k}}
e^{i\lambdabar\big(1+\frac{\Phi(z,z',v)}{t-s}\big)} L^N\Big(\lambdabar^{n-1+k/2}
\varphi(\frac{\lambdabar}{t-s})a(\frac{\lambdabar}{t-s},y,y',\sigma,v) \phi_m(\lambdabar) \Big)\,
dv \, d\lambdabar \\ &\lesssim (t-s)^{-n} 2^{-m(N-n - k/2)}
\end{split} \label{intt}
\end{equation*}
Hence we prove \eqref{k=0} by summing over
$m\geq0$, thus \eqref{UiUj2} follows.

\end{proof}

\subsection{Abstract Stirchartz estimate on Lorentz space}

To prove the Strichartz estimate, we sharpen the semiclassical version of Strichartz
estimates \cite[Proposition 4.1]{Z2} to Lorentz space $L^{r,2}$ by following abstract Keel-Tao's Strichartz estimates
theorem.

\begin{proposition}\label{prop:semi}
Let $(X,\mathcal{M},\mu)$ be a $\sigma$-finite measured space and
$U: \mathbb{R}\rightarrow B(L^2(X,\mathcal{M},\mu))$ be a weakly
measurable map satisfying, for some constants $C$, $\alpha\geq0$,
$\sigma, h>0$,
\begin{equation}\label{md}
\begin{split}
\|U(t)\|_{L^2\rightarrow L^2}&\leq C,\quad t\in \mathbb{R},\\
\|U(t)U(s)^*f\|_{L^\infty}&\leq
Ch^{-\alpha}(h+|t-s|)^{-\sigma}\|f\|_{L^1}.
\end{split}
\end{equation}
Then for every pair $q,\rr\in[2,\infty]$ such that $(q,\rr,\sigma)\neq
(2,\infty,1)$ and
\begin{equation*}
\frac{1}{q}+\frac{\sigma}{\rr}\leq\frac\sigma 2,\quad q\ge2,
\end{equation*}
there exists a constant $\tilde{C}$ only depending on $C$, $\sigma$,
$q$ and $r$ such that
\begin{equation}\label{s-stri}
\Big(\int_{\R}\|U(t) u_0\|_{L^{\rr,2}}^q dt\Big)^{\frac1q}\leq \tilde{C}
\Lambda(h)\|u_0\|_{L^2}
\end{equation}
where $\Lambda(h)=h^{-(\alpha+\sigma)(\frac12-\frac1\rr)+\frac1q}$.
\end{proposition}

\begin{proof} For convenience, we write down the proof by repeating the argument in \cite{Z2} but with minor modification of the interpolation. If
$(q,\rr,\sigma)\neq (2,\infty,1)$ is on the line $\frac1q+\frac\sigma
\rr=\frac\sigma 2$, we replace $(|t-s|+h)^{-\sigma}$ by
$|t-s|^{-\sigma}$ and then we closely follow Keel-Tao's argument
\cite[Sections 3-7]{KT} to show \eqref{s-stri}. We remark here that the alternative interpolation argument in \cite [Section 6]{KT} shows the inequalities sharpened to Lorentz space. So we only consider
$\frac1q+\frac\sigma \rr<\frac\sigma 2$. By the $TT^*$ argument, it
suffices to show
\begin{equation*}
\begin{split}
\Big|\iint\langle U(s)^*f(s), U(t)^*g(t) \rangle dsdt\Big|\lesssim
\Lambda(h)^2\|f\|_{L^{q'}_tL^{\rr',2}}\|g\|_{L^{q'}_tL^{\rr',2}}.
\end{split}
\end{equation*}
Using the bilinear interpolation of \eqref{md} in \cite[Lemma 6.1]{KT}, we have
\begin{equation*}
\begin{split}
\langle U(s)^*f(s), U(t)^*g(t) \rangle&\leq
Ch^{-\alpha(1-\frac2\rr)}(h+|t-s|)^{-\sigma(1-\frac2\rr)}\|f\|_{L^{\rr',2}}\|g\|_{L^{\rr',2}}.
\end{split}
\end{equation*}
Therefore, we see by H\"older's and Young's inequalities for
$\frac1q+\frac\sigma \rr<\frac\sigma 2$
\begin{equation*}
\begin{split}
\Big|\iint\langle U(s)^*f(s),& U(t)^*g(t) \rangle
dsdt\Big|\\&\lesssim
h^{-\alpha(1-\frac2\rr)}\iint(h+|t-s|)^{-\sigma(1-\frac2\rr)}\|f(t)\|_{L^{\rr',2}}\|g(s)\|_{L^{\rr',2}}dtds\\&
\lesssim
h^{-\alpha(1-\frac2\rr)}h^{-\sigma(1-\frac2\rr)+\frac2q}\|f\|_{L^{q'}_tL^{\rr',2}}\|g\|_{L^{q'}_tL^{\rr',2}}.
\end{split}
\end{equation*}
This proves \eqref{s-stri}.
\end{proof}

\subsection{Homogeneous Strichartz estimates} Now we show the homogeneous Strichartz estimate.
Let $u$ solve \begin{equation}\label{leq}
\partial_{t}^2u+\LL_0 u=0, \quad u(0)=u_0,
~\partial_tu(0)=u_1,
\end{equation}
then for $q,\rr\geq2$, the square function
estimates \eqref{squareL} and Minkowski's inequality show that
\begin{equation}\label{LP}
\|u\|_{L^q(\R;L^{\rr,2}(X))}\lesssim
\Big(\sum_{k\in\Z}\|u_k\|^2_{L^q(\R;L^{\rr,2}(X))}\Big)^{\frac12}
\end{equation}
where $u_k$ is defined by
\begin{equation}\label{loc}
u_k(t,\cdot)=\varphi(2^{-k}\sqrt{\LL_0})u(t,\cdot),
\end{equation}
where $\varphi$ is as in \eqref{dp}.
Applying the operator $\varphi(2^{-k}\sqrt{\LL_0})$ to the wave equation, we obtain
\begin{equation}\label{leq}
\partial_{t}^2u_k+\LL_0 u_k=0, \quad u_k(0)=f_k(z),
~\partial_tu_k(0)=g_k(z),
\end{equation}
where $f_k=\varphi(2^{-k}\sqrt{\LL_0})u_0$ and
$g_k=\varphi(2^{-k}\sqrt{\LL_0})u_1$. Let $U(t)=e^{it\sqrt{\LL_0}}$, then
we write
\begin{equation}\label{sleq}
\begin{split}
u_k(t,z)
=\frac{U(t)+U(-t)}2f_k+\frac{U(t)-U(-t)}{2i\sqrt{\LL_0}}g_k.
\end{split}
\end{equation}
For our purpose, we need the following
\begin{proposition}\label{lStrichartz} Let
$f=\varphi(2^{-k}\sqrt{\LL_0})f$ for $k\in\Z$ and $U(t)=e^{it\sqrt{\LL_0}}$, we have
\begin{equation}\label{lstri}
\|U(t)f\|_{L^q_tL^{\rr,2}_z(\mathbb{R}\times X)}\lesssim
2^{ks}\|f\|_{L^2(X)},
\end{equation}
where the admissible pair $(q,\rr)\in [2,\infty]^2$ satisfies
\eqref{adm} and $s=n(\frac12-\frac1\rr)-\frac1q$.
\end{proposition}
\begin{proof}Let $\alpha=(n+1)/2$, $\sigma=(n-1)/2$ and
$h=2^{-k}$, by Proposition \ref{prop:Dispersive}, we have the estimates \eqref{md}
for $U_{j,k}(t)$. Then it follows from Proposition \ref{prop:semi} that
\begin{equation*}
\|U_{j,k}(t)f\|_{L^q_t(\R:L^{\rr,2}(X))}\lesssim
2^{k[n(\frac12-\frac1\rr)-\frac1q]} \|f\|_{L^2(X)}.
\end{equation*}
On the other hand, we have
\begin{equation*}
U(t)=\sum_{j=0}^{N}\sum_{k\in\Z}U_{j,k}(t).
\end{equation*}
Let $\widetilde{\varphi} \in C_0^\infty(\R\setminus\{0\})$ take
values in $[0,1]$ such that $\widetilde{\varphi}\varphi=\varphi$, hence we can write
\begin{equation*}
U(t)f=\sum_{j}\sum_{k\in\mathbb{Z}}\int_0^\infty
e^{it\lambda}\varphi(2^{-k}\lambda)Q_j(\lambda)dE_{\sqrt{\LL_0}}(\lambda)
\widetilde{\varphi}(2^{-k}\sqrt{\LL_0})f.
\end{equation*}
Notice $f=\varphi(2^{-k}\sqrt{\LL_0})f$, then
$\widetilde{\varphi}(2^{-k'}\sqrt{\LL_0})f$ vanishes if
$|k-k'|\gg1$. Hence we obtain
\begin{equation*}
\|U(t)f\|_{L^q_t(\R:L^{\rr,2}(X))}\lesssim
2^{k[n(\frac12-\frac1\rr)-\frac1q]} \|f\|_{L^2(X)}.
\end{equation*}
Therefore, we prove this proposition.
\end{proof}
By \eqref{LP} and \eqref{sleq} and \eqref{lstri}, we have that
\begin{equation*}
\begin{split}
&\|u\|_{L^q(\R;L^{\rr,2}(X))}\\&\lesssim
\Big(\sum_{k\in\Z}\big(2^{2ks}\|\varphi(2^{-k}\sqrt{\LL_0})u_0\|^2_{L^2(X)}+2^{2k(s-1)}\|\varphi(2^{-k}\sqrt{\LL_0})u_1\|^2_{L^2(X)}\big)\Big)^{\frac12}.
\end{split}
\end{equation*}
By Littlewood-Paley theory again \eqref{squareL}, we prove
\begin{equation*}
\|u\|_{L^q(\R;L^{\rr,2}(X))}\lesssim
\|u_0\|_{\dot H^s(X)}+\|u_1\|_{\dot H^{s-1}(X)}.
\end{equation*}

\subsection{Inhomogeneous Strichartz estimates}\label{4.5}
In this subsection, we derive the inhomogeneous Strichartz estimate from the homogeneous Strichartz estimate by using Christ-Kiselev lemma \cite{CK}.  Recall the half-wave operator
$U(t)=e^{it\sqrt{\LL_0}}: L^2\rightarrow L^2$ and in last subsection we have just
proved that
\begin{equation}
\|U(t)u_0\|_{L^q_tL^{\rr,2}_z}\lesssim\|u_0\|_{\dot{H}^s}
\end{equation} holds for all $(q,\rr,s)$ satisfying \eqref{adm} and \eqref{scaling}.
Given $s\in\R$ and $(q,r)\in\Lambda_s$, define the operator ${\bf T}_s$ by
\begin{equation}\label{Ts}
\begin{split}
{\bf T}_s: L^2_z&\rightarrow L^q_tL^{\rr,2}_z,\quad f\mapsto \LL_0^{-\frac
s2}e^{it\sqrt{\LL_0}}f.
\end{split}
\end{equation}
By the dual of Lorentz space in Proposition \ref{prop:dual},  we have
\begin{equation}\label{Ts*}
\begin{split}
{\bf T}^*_{1-s}: L^{\tilde{q}'}_tL^{\tilde{\rr}',2}_z\rightarrow L^2,\quad
F(\tau,z)&\mapsto \int_{\R}\LL_0^{\frac
{s-1}2}e^{-i\tau\sqrt{\LL_0}}F(\tau)d\tau,
\end{split}
\end{equation}
where $1-s=n(\frac12-\frac1{\tilde{\rr}})-\frac1{\tilde{q}}$.
It shows that
\begin{equation*}
\Big\|\int_{\R}U(t)U^*(\tau)\LL_0^{-\frac12}F(\tau)d\tau\Big\|_{L^q_tL^{\rr,2}_z}
=\big\|{\bf T}_s{\bf T}^*_{1-s}F\big\|_{L^q_tL^{\rr,2}_z}\lesssim\|F\|_{L^{\tilde{q}'}_tL^{\tilde{\rr}',2}_z}.
\end{equation*}
Note that $s=n(\frac12-\frac1\rr)-\frac1q$ and
$1-s=n(\frac12-\frac1{\tilde{\rr}})-\frac1{\tilde{q}}$, thus $(q,\rr),
(\tilde{q},\tilde{\rr})$ satisfy \eqref{scaling}. By the
Christ-Kiselev lemma \cite{CK}, we thus obtain for $q>\tilde{q}'$,
\begin{equation}\label{non-inhomgeneous}
\begin{split}
\Big\|\int_{\tau<t}\frac{\sin{(t-\tau)\sqrt{\LL_0}}}
{\sqrt{\LL_0}}F(\tau)d\tau\Big\|_{L^q_tL^{\rr,2}_z}\lesssim\|F\|_{L^{\tilde{q}'}_t{L}^{\tilde{\rr}',2}_z}.
\end{split}
\end{equation}
Notice that for all $(q,\rr), (\tilde{q},\tilde{\rr})\in\Lambda_s$, one must have $q>\tilde{q}'$.\vspace{0.2cm}

Therefore, we conclude that:
\begin{proposition}\label{Str-L0} For any $s\in\R$, let $(q,\rr), (\tilde{q},\tilde{\rr})\in \Lambda_s$ and let $u$ be the solution to
\begin{equation}\label{leq}
\partial_{t}^2u+\LL_0 u=F, \quad u(0)=u_0,
~\partial_tu(0)=u_1,
\end{equation}
the following Strichartz estimates hold:
\begin{equation}\label{Str-L0-est}
\|u(t,z)\|_{L^q(\R;L^{\rr,2}(X))}\leq C\left(\|u_0\|_{\dot H^s(X)}+\|u_1\|_{\dot H^{s-1}(X)}+\|F\|_{L^{\tilde{q}'}(\R;L^{\tilde{\rr}',2}(X))}\right).
\end{equation}

\end{proposition}

\begin{remark} This result concludes the full range set of global-in-time Strichartz estimates both in homogenous and inhomogeneous inequalities when $V=0$.
Hence, by embedding inequality of Lorentz space, we prove Theorem \ref{thm:Strichartz} when $V=0$.
\end{remark}

\section{Inhomogeneous Strichartz estimates with $q=\tilde{q}=2$} In the next section, we need the following result on the double endpoint  inhomogeneous Strichartz estimate.
\begin{proposition}\label{prop:inh} Let $\rr={2(n-1)}/{(n-3)}$ and $F=\varphi(2^{-k}\sqrt{\LL_0})F$, we have the following inequality
\begin{equation}\label{inh}
\begin{split}
\Big\|\int_{\tau<t}\frac{\sin{(t-\tau)\sqrt{\LL_0}}}
{\sqrt{\LL_0}}F(\tau)d\tau\Big\|_{L^2_tL^{\rr,2}_z}\lesssim 2^{k[2n(\frac12-\frac1\rr)-2]}\|F\|_{L^{2}_t{L}^{\rr',2}_z}.
\end{split}
\end{equation}
\end{proposition}
As a consequence, we have
\begin{corollary}\label{cor:inh} Let $\rr={2(n-1)}/{(n-3)}$, the following inequality holds
\begin{equation}\label{inh}
\begin{split}
\Big\|\LL_0^{-\frac1{n-1}}\int_{\tau<t}\frac{\sin{(t-\tau)\sqrt{\LL_0}}}
{\sqrt{\LL_0}}F(\tau)d\tau\Big\|_{L^2_tL^{\rr,2}_z}\lesssim \|F\|_{L^{2}_t{L}^{\rr',2}_z}.
\end{split}
\end{equation}
\end{corollary}
\begin{proof}
This is a consequence of the Littlewood-Paley theory in Lemma \ref{prop:square}.
\end{proof}
\begin{remark}
This inhomogeneous inequality is not included in the above estimate\eqref{Str-L0-est} since if $q=\tilde{q}=2$, then at least, one of $(q,\rr), (\tilde{q},\tilde{\rr})$
is not in $\Lambda_s$. However, if we only consider the inhomogeneous
Strichartz estimate, we can obtain this endpoint estimate \eqref{inh} by following  the argument of \cite{KT} and \cite{HZ},
although at this moment we only have the microlocalized dispersive estimates \eqref{UiUjnear}-\eqref{UiUj2}.
For more inhomogeneous estimates, we refer the reader to \cite{Foschi, Vilela}
where the propagator satisfies the classical dispersive estimate.
\end{remark}

\begin{proof} Recall $U(t)=e^{it\sqrt{\LL_0}}$, then
$$\frac{\sin{(t-\tau)\sqrt{\LL_0}}}
{\sqrt{\LL_0}}=\LL_0^{-\frac12}(U(t)U(\tau)^*-U(-t)U(-\tau)^*)/2i.$$
Hence to show \eqref{inh},  it suffices to show the bilinear form estimate
\begin{equation}\label{TFG}
|T_{k}(F,G)|\leq 2^{2k[n(\frac12-\frac1\rr)-\frac12]} \|F\|_{L^{2}_tL^{\rr',2}_z}\|G\|_{L^{2}_tL^{\rr',2}_z},
\end{equation}
where $\rr=2(n-1)/(n-3)$ and $T_k(F,G)$ is the bilinear form
\begin{equation}
T_{k}(F,G)=\iint_{s<t}\langle U_{k}(t)U_{k}^*(\tau)F(\tau), G(t)\rangle_{L^2}~ d\tau dt
\end{equation}
where $U_k=\sum_{0\leq j\leq N} U_{j,k}$ defined in \eqref{Uk}.

On the other hand, we have proved that for all $(q,\rr)\in\Lambda_s$ with $s=n(\frac12-\frac1\rr)-\frac1q$
\begin{equation*}
\|U_{j,k}(t)f\|_{L^2_t(\R:L^{\rr,2}(X))}\lesssim
2^{k[n(\frac12-\frac1\rr)-\frac1q]} \|f\|_{L^2(X)}.
\end{equation*}
By duality, we have
\begin{equation*}
\Big\|\int_{\R}U_{j,k}(t)U_{j',k}^*(\tau)F(\tau)d\tau\Big\|_{L^q_tL^{\rr,2}_z}\lesssim 2^{2k[n(\frac12-\frac1\rr)-\frac1q]}
\|F\|_{L^{q'}_\tau L^{\rr',2}_z}, \forall 0\leq j,j'\leq N.
\end{equation*}
In particular $q=2, \rr=\frac{2(n-1)}{n-3}$, it follows that for all $0\leq j,j'\leq N$,
\begin{equation}\label{R2}
\iint_{\R^2}\langle U_{j,k}(t)U_{j',k}^*(\tau)F(\tau), G(t)\rangle_{L^2}~ d\tau dt\leq
C 2^{2k[n(\frac12-\frac1\rr)-\frac12]} \|F\|_{L^2_\tau L^{\rr',2}_z}\|G\|_{L^2_tL^{\rr',2}_z}.
\end{equation}

We need the following bilinear estimates

\begin{lemma}\label{either} Let $U_{j,k}(t)$ be defined as in \eqref{Uitj}, then for each pair $(j,j')\in \{ 0, 1, \dots, N \}^2 $
there exists a constant $C$ such that, for each $k$, either
\begin{equation}\label{bilinear:s<t}
\iint_{\tau<t}\langle U_{j,k}(t)U_{j',k}^*(\tau)F(\tau), G(t)\rangle_{L^2}~ d\tau dt\leq C 2^{2k[n(\frac12-\frac1\rr)-\frac12]}
\|F\|_{L^2_\tau L^{\rr',2}_z}\|G\|_{L^2_tL^{\rr',2}_z},
\end{equation}
or
\begin{equation}\label{bilinear:s>t}
\iint_{\tau>t}\langle U_{j,k}(t)U_{j',k}^*(\tau)F(\tau), G(t)\rangle_{L^2}~ d\tau dt\leq C 2^{2k[n(\frac12-\frac1\rr)-\frac12]}
\|F\|_{L^2_\tau L^{\rr',2}_z}\|G\|_{L^2_tL^{\rr',2}_z}.
\end{equation}
\end{lemma}
We postpone the proof for a moment. Hence for every pair $(j,j')$, we have by \eqref{bilinear:s<t} or subtracting \eqref{bilinear:s>t} from \eqref{R2}
\begin{equation*}
\iint_{\tau<t}\langle U_{j,k}(t)U_{j',k}^*(\tau)F(\tau), G(t)\rangle_{L^2}~ d\tau dt\leq C 2^{2k[n(\frac12-\frac1\rr)-\frac12]}
\|F\|_{L^2_\tau L^{\rr',2}_z}\|G\|_{L^2_tL^{\rr',2}_z}.
\end{equation*}
Finally by summing over all $j$ and $j'$, we obtain \eqref{TFG}. Once we prove Lemma \ref{either}, we complete the proof of Proposition \ref{prop:inh}.

\end{proof}

\begin{proof}[Proof of Lemma~\ref{either}] Without loss of generality, by scaling argument, we may assume $k=0$. In the case that $(j,j') \in J_{near}$ or $(j,j')=(j,0)$ or $(j,j')=(0,j')$,
we have the dispersive estimate \eqref{UiUjnear}. We
apply the argument of \cite[Sections 4--7]{KT} to obtain
\eqref{bilinear:s<t}. If $(j,j') \in J_{non-out}$, we
obtain \eqref{bilinear:s<t} adapting the argument in \cite{KT}
due to the dispersive estimate \eqref{UiUj2} when $\tau < t$.
Finally, in the case that $(j,j') \in J_{non-inc}$, we obtain
\eqref{bilinear:s>t} since we have the dispersive estimate
\eqref{UiUj} for $\tau> t$. We mention here that we have sharpened the inequality to the Lorentz norm by the interpolation as remarked in \cite[Section 6 and Section 10]{KT}.
\end{proof}

\section{Strichartz estimates for wave equation with $\LL_V$}\label{subsec:4.2}
In this section, we prove the Strichartz estimate for $\LL_V$ by using Proposition \ref{Str-L0} and establishing a local smoothing estimate.

\subsection{A local-smoothing estimate}

In this subsection, we prove a global-in-time local-smoothing estimate.
It worths pointing out that we directly prove the local smoothing estimate avoiding the usual method via resolvent estimate of $\LL_V$.

\begin{proposition}\label{prop:loc} Let $u$ be the solution of \eqref{equ:wave}, then there exists a constant $C$
independent of $(u_0,u_1)$ such that
\begin{equation}\label{local-s}
\begin{split}
\|r^{-\beta}u(t,z)\|_{L^2_t(\R;L^2(X))}\leq
C\left(\|u_0\|_{\dot H^{\beta-\frac12}(X)}+\|u_1\|_{\dot H^{\beta-\frac32}(X)}\right),
\end{split}
\end{equation}
where $z=(r,y)\in X$, $1/2<\beta<1+\nu_0$ with $\nu_0>0$ such that $\nu_0^2$ is the smallest eigenvalue of $\Delta_h+V_0(y)+(n-2)^2/4$.
\end{proposition}
\begin{remark} In \cite{BPST}, Burq et al. established the resolvent estimate and thus proved a same estimate, called Morawetz estimate,
in Euclidean space with $\beta=1$.
\end{remark}

\begin{proof}We modify the proof of the argument in our previous paper \cite{ZZ1} for Schr\"odinger.
A key observation is that the norms in the both sides of the local smoothing are based on
$L^2$-space which allows us to use orthogonality of eigenfunctions. Without loss of generality, we assume $u_1=0$. Since
\begin{equation}\label{sleq}
\begin{split}
u(t,z)
=\frac12\left(e^{it\sqrt{\LL_V}}+e^{-it\sqrt{\LL_V}}\right)u_0,
\end{split}
\end{equation}
we only consider the estimate of $e^{it\sqrt{\LL_V}}u_0$.
Recall $$
u_0(z)=\sum_{\nu\in\chi_\infty}\sum_{\ell=1}^{d(\nu)}a_{\nu,\ell}(r)\varphi_{\nu,\ell}(y), \quad  b_{\nu,\ell}(\rho)=(\mathcal{H}_{\nu}a_{\nu,\ell})(\rho).
$$
By \eqref{funct} with $F(\rho)=e^{it\rho}$, we will estimate
\begin{equation}
\begin{split}
e^{it\sqrt{\LL_V}}u_0&=\sum_{\nu\in\chi_\infty}\sum_{\ell=1}^{d(\nu)}\varphi_{\nu,\ell}(\theta)\int_0^\infty(r\rho)^{-\frac{n-2}2}J_{\nu}(r\rho)e^{
it\rho}b_{\nu,\ell}(\rho)\rho^{n-1}d\rho.
\end{split}
\end{equation}
By the Plancherel theorem with respect to time $t$, it suffices to estimate
\begin{equation*}
\begin{split}
\int_{X}\int_0^\infty\big|\sum_{\nu\in\chi_\infty}\sum_{\ell=1}^{d(\nu)}\varphi_{\nu,\ell}(\theta)(r\rho)^{-\frac{n-2}2}J_{\nu}(r\rho)
b_{\nu,\ell}(\rho)\rho^{n-1}\big|^2d\rho r^{-2\beta}d\mu(z)
\end{split}
\end{equation*}
Using the orthogonality, one has
\begin{equation*}
\begin{split}
\int_{Y} \big|\sum_{\nu\in\chi_\infty}\sum_{\ell=1}^{d(\nu)}\varphi_{\nu,\ell}(\theta)J_{\nu}(r\rho)b_{\nu,\ell}(\rho)
\big|^2 d\theta=\sum_{\nu\in\chi_\infty}\sum_{\ell=1}^{d(\nu)}\big|J_{\nu}(r\rho)b_{\nu,\ell}(\rho)
\big|^2
\end{split}
\end{equation*}
then we see that the above is equal to
\begin{equation*}
\begin{split}
\sum_{\nu\in\chi_\infty}\sum_{\ell=1}^{d(\nu)}\int_0^\infty\int_0^\infty\big|(r\rho)^{-\frac{n-2}2}J_{\nu}(r\rho)b_{\nu,\ell}(\rho)
\rho^{n-1}\big|^2 d\rho r^{n-1-2\beta}dr.
\end{split}
\end{equation*}
To estimate it, we make a dyadic decomposition into the integral. Let $\chi$ be a smoothing function supported in $[1,2]$, we see that the above is less than
\begin{equation}\label{scal-reduce}
\begin{split}
&\sum_{\nu\in\chi_\infty}\sum_{\ell=1}^{d(\nu)}\sum_{M\in2^{\Z}}\int_0^\infty\int_0^\infty\big|(r\rho)^{-\frac{n-2}2}J_{\nu}(r\rho)b_{\nu,\ell}(\rho)
\rho^{n-1}\chi(\frac{\rho}{M})\big|^2d\rho r^{n-1-2\beta}dr\\&\lesssim
\sum_{\nu\in\chi_\infty}\sum_{\ell=1}^{d(\nu)}\sum_{M\in2^{\Z}}\sum_{R\in2^{\Z}}M^{n-1+2\beta}R^{n-1-2\beta}\int_{R}^{2R}
\int_{0}^\infty\big|(r\rho)^{-\frac{n-2}2}J_{\nu}(r\rho)b_{\nu,\ell}(M\rho)\chi(\rho)
\big|^2 d\rho dr.
\end{split}
\end{equation}
Let
\begin{equation}\label{def:G}
\begin{split}
Q_{\nu,\ell}(R,M)=\int_{R}^{2R}\int_{0}^\infty\big|(r\rho)^{-\frac{n-2}2}J_{\nu}(r\rho)b_{\nu,\ell}(M\rho)\chi(\rho)
\big|^2 d\rho  dr.\end{split}
\end{equation}
Then we have the following inequality
\begin{equation}\label{est:Q}
Q_{\nu,\ell}(R,M) \lesssim
\begin{cases}
R^{2\nu-n+3}M^{-n}\|b_{\nu,\ell}(\rho)\chi(\frac{\rho}M)\rho^{\frac{n-1}2}\|^2_{L^2},~
R\lesssim 1;\\
R^{-(n-2)}M^{-n}\|b_{\nu,\ell}(\rho)\chi(\frac{\rho}M)\rho^{\frac{n-1}2}\|^2_{L^2},~
R\gg1.
\end{cases}
\end{equation}
We postpone the proof for a moment.  By \eqref{est:Q} we turn to estimate
\begin{equation*}
\begin{split}
&\sum_{\nu\in\chi_\infty}\sum_{\ell=1}^{d(\nu)}\sum_{M\in2^{\Z}}\int_0^\infty\int_0^\infty\big|(r\rho)^{-\frac{n-2}2}J_{\nu}(r\rho)b_{\nu,\ell}(\rho)
\rho^{n-1}\chi(\frac{\rho}{M})\big|^2 d\rho r^{n-1-2\beta} dr\\&\lesssim
\sum_{\nu\in\chi_\infty}\sum_{\ell=1}^{d(\nu)}\sum_{M\in2^{\Z}}\sum_{R\in2^{\Z}}M^{n-1+2\beta}R^{n-1-2\beta}Q_{\nu,\ell}(R,M)\\&\lesssim
\sum_{\nu\in\chi_\infty}\sum_{\ell=1}^{d(\nu)}\sum_{M\in2^{\Z}}\Big(\sum_{R\in2^{\Z}, R\lesssim 1}M^{n-1+2\beta}R^{n-1-2\beta}R^{2\nu-n+3}M^{-n}
\\&\quad\qquad+\sum_{R\in2^{\Z}, R\gg 1}M^{n-1+2\beta}R^{n-1-2\beta}R^{-(n-2)}M^{-n}\Big)\|b_{\nu,\ell}(\rho)\chi(\frac{\rho}M)\rho^{\frac{n-1}2}\|^2_{L^2}
\\&\lesssim
\sum_{\nu\in\chi_\infty}\sum_{\ell=1}^{d(\nu)}\sum_{M\in2^{\Z}}\Big(\sum_{R\in2^{\Z}, R\lesssim 1}M^{2\beta-1}R^{2(1+\nu-\beta)}
+\sum_{R\in2^{\Z}, R\gg 1}M^{2\beta-1}R^{1-2\beta}\Big)\|b_{\nu,\ell}(\rho)\chi(\frac{\rho}M)\rho^{\frac{n-1}2}\|^2_{L^2}.
\end{split}
\end{equation*}
Note that if $\frac12<\beta<1+\nu_0 $ the summations in $R$ converges  and further converges to $\|u_0\|^2_{\dot H^{\beta-\frac12}(X)}$. Hence we prove \eqref{local-s}.
Now we are left to prove \eqref{est:Q}. To this end, we break it into two
cases.\vspace{0.2cm}

$\bullet$ Case 1: $R\lesssim1$. Since $\rho\sim1$, thus
$r\rho\lesssim1$. By
\eqref{bessel-r}, we obtain
\begin{equation*}
\begin{split}
Q_{\nu,\ell}(R,M)&\lesssim\int_{R}^{2R}\int_{0}^\infty\Big| \frac{
(r\rho)^{\nu}(r\rho)^{-\frac{n-2}2}}{2^{\nu}\Gamma(\nu+\frac12)\Gamma(\frac12)}b_{\nu,\ell}(M\rho)\chi(\rho)\Big|^2 d\rho dr\\& \lesssim
R^{2\nu-n+3}M^{-n}\|b_{\nu,\ell}(\rho)\chi(\frac{\rho}M)\rho^{\frac{n-1}2}\|^2_{L^2}.
\end{split}
\end{equation*}

$\bullet$ Case 2: $R\gg1$. Since $\rho\sim1$, thus $r\rho\gg 1$. We
estimate by \eqref{est:b} in Lemma \ref{lem: J}
\begin{equation*}
\begin{split}
Q_{\nu,\ell}(R,M)&\lesssim
R^{-(n-2)}\int_{0}^\infty\big|b_{\nu,\ell}(M\rho)\chi(\rho)\big|^2\int_{R}^{2R}\big|J_{\nu}(
r\rho)\big|^2  dr d\rho\\& \lesssim R^{-(n-2)}\int_{0}^\infty\big|b_{\nu,\ell}(M\rho)\chi(\rho)\big|^2 d\rho\lesssim
R^{-(n-2)}M^{-n}\|b_{\nu,\ell}(\rho)\chi(\frac{\rho}M)\rho^{\frac{n-1}2}\|^2_{L^2}.
\end{split}
\end{equation*}
Thus we prove \eqref{est:Q}. Therefore, we prove the local smoothing estimate.

\end{proof}

\begin{remark}\label{rem:loc-s} By constructing the similar counterexample as in Subsection 6.3, we can see the restriction $\beta<1+\nu_0$ is necessary for \eqref{local-s}. However,

(i) if $V=0$, then $\beta<\frac n2$ is required since the positive square root of the smallest eigenvalue of $\Delta_h+(n-2)^2/4$ is greater than $(n-2)/2$;

(ii) if the initial data, say $u_0$, belongs to $\bigoplus_{\nu\in\chi_\infty,\nu>k} \mathcal{H}^{\nu}\cap\dot H^{\beta-\frac12}(X)$ where $k>\nu_0$, then one
can relax the restriction on $\beta$ to  $\beta<1+k$.

\end{remark}

\subsection{The proof of Strichartz estimates }
Let $v$ be as in Proposition \ref{Str-L0} with $F=0$ and suppose that $u$ solves the equation
\begin{equation*}
\partial_{t}^2u+\LL_V u=0, \quad u(0)=u_0,
~\partial_tu(0)=u_1,
\end{equation*}
we have by the Duhamel formula
\begin{equation}\label{duhamel}
\begin{split}
u(t,z)&=\frac{e^{it\sqrt{\LL_V}}+e^{-it\sqrt{\LL_V}}}2 u_0+\frac{e^{it\sqrt{\LL_V}}-e^{-it\sqrt{\LL_V}}}{2i\sqrt{\LL_V}}u_1\\&=v(t,z)+\int_0^t\frac{\sin{(t-\tau)\sqrt{\LL_0}}}
{\sqrt{\LL_0}}(V(z)u(\tau,z))d\tau.
\end{split}
\end{equation}
From the spectral theory on $L^2$, we have the Strichartz estimate for $(q,\rr)=(\infty, 2)$. By using the Sobolev inequality in Proposition \ref{P:sobolev} ,
we obtain
\begin{equation*}
\begin{split}
\|u(t,z)\|_{L^\infty(\R;L^{\rr}(X))}&\lesssim \|\LL^{\frac s2}_V u(t,z)\|_{L^\infty(\R;L^{2}(X))}\\
&\lesssim \|u_0\|_{\dot H^{s}(X)}+\|u_1\|_{\dot H^{s-1}(X)}\end{split}
\end{equation*}
where $s=n(1/2-1/\rr)<2$ and $2\leq \rr<n/\max\{\frac n2-1-\nu_0,  0\}$. Note that the restriction $s<2$ implies $\rr<2n/(n-4)$
which is a artificial restriction, thus we can get rid of this restriction by using an iterating argument
as in Corollary \ref{P:sobolev'} .

If $\nu_0>1/(n-1)$, then we have
\begin{equation}2\leq \rr< \begin{cases} \infty, \quad &n=3;\\
n/\max\{\frac n2-1-\nu_0,0\},\quad & n\geq4,
\end{cases}
\end{equation}
which is corresponding to $0<s<1+\nu_0$.
On the other hand, by Proposition \ref{Str-L0} with $s\in\R$ and H\"older's inequality in Proposition \ref{Lorentz}, we show that
\begin{equation*}
\begin{split}
&\|u(t,z)\|_{L^q(\R;L^{\rr}(X))}\\&\lesssim \|u_0\|_{\dot H^{s}(X)}+\|u_1\|_{\dot H^{s-1}(X)}+\Big\|\int_0^t\frac{\sin{(t-\tau)\sqrt{\LL_0}}}
{\sqrt{\LL_0}}(V(z)u(\tau,z))d\tau\Big\|_{_{L^q(\R;L^{\rr}(X))}}
\end{split}
\end{equation*}

Now our main task is to estimate
\begin{equation}\label{est:inh}
\begin{split}
\Big\|\int_0^t\frac{\sin{(t-\tau)\sqrt{\LL_0}}}
{\sqrt{\LL_0}}(V(z)u(\tau,z))d\tau\Big\|_{_{L^q(\R;L^{\rr}(X))}}.
\end{split}
\end{equation}
Note that if the set $\Lambda_s$ is not empty, we must have $s\geq0$. Indeed, if $(q,\rr)\in\Lambda_s$, then
\begin{equation}\label{s:index}
s=n\Big(\frac12-\frac1\rr\Big)-\frac1q\geq\frac12(n+1)\Big(\frac12-\frac1\rr\Big)\geq0.
\end{equation}
Therefore, without loss of generality,  we may assume
$s>0$. \vspace{0.2cm}

Now we argue Theorem \ref{thm:Strichartz} by considering the following four cases.\vspace{0.1cm}

{\bf Case I: $ 0<s<\frac12+\nu_0, q>2$. }
Let $\frac12<\beta<n/2$, by using Proposition \ref{prop:loc} and Remark \ref{rem:loc-s}, we define the operator
$$T: L^2(X)\to L^2(\R;L^2(X)), \quad T f= r^{-\beta}e^{it\sqrt{\LL_0}}\LL_0^{\frac12(\frac12-\beta)} f.$$
Thus from the proof of the local smoothing estimate, it follows that $T$ is a bounded operator.
By the duality, we obtain that for its adjoint $T^*$ $$T^*: L^2(\R;L^2(X))\to L^2, \quad
T^* F=\int_{\tau\in\R}\LL_0^{\frac12(\frac12-\beta)} e^{-i\tau\sqrt{\LL_0}}  r^{-\beta}  F(\tau)d\tau$$
which is also bounded. Define the operator
$$B: L^2(\R;L^2(X))\to L^q(\R;L^r(X)), \quad B F=\int_{\tau\in\R} \frac{e^{i(t-\tau)\sqrt{\LL_0}}}{\sqrt{\LL_0}} r^{-\beta}F(\tau)d\tau.$$
Hence by the Strichartz estimate with $s=\frac32-\beta$, one has
\begin{equation}\label{BF}
\begin{split}
&\|B F\|_{L^q(\R;L^\rr(X))}\\&=\big\| e^{i t\sqrt{\LL_0}}\int_{\tau\in\R}\frac{e^{-i\tau\sqrt{\LL_0}}}{\sqrt{\LL_0}} r^{-\beta} F(\tau)d\tau\big\|_{L^q(\R;L^\rr(X))}\\
&\lesssim \big\|\int_{\tau\in\R}\frac{e^{-i\tau\sqrt{\LL_0}}}{\sqrt{\LL_0}} r^{-\beta} F(\tau)d\tau\big\|_{\dot H^{\frac32-\beta}(X)}
=\|T^*F\|_{L^2}\lesssim \|F\|_{L^2(\R;L^2(X))}.
\end{split}
\end{equation}
Now we estimate \eqref{est:inh}. Note that
$$\sin(t-\tau)\sqrt{\LL_0}=\frac{1}{2i}\big(e^{i(t-\tau)\sqrt{\LL_0}}-e^{-i(t-\tau)\sqrt{\LL_0}}\big),$$
thus by \eqref{BF}, we have a minor modification of \eqref{est:inh}
\begin{equation*}
\begin{split}
&\Big\|\int_\R\frac{\sin{(t-\tau)\sqrt{\LL_0}}}
{\sqrt{\LL_0}}(V(z)u(\tau,z))d\tau\Big\|_{L^q(\R;L^{\rr}(X))}\\&\lesssim \|B(r^{\beta}V(z)u(\tau,z))\|_{L^q(\R;L^\rr(X))}\lesssim \|r^{\beta-2}u(\tau,z))\|_{L^2(\R;L^2(X))}
\\&\lesssim \|u_0\|_{\dot H^{\frac32-\beta}(X)}+\|u_1\|_{\dot H^{\frac12-\beta}(X)}
\end{split}
\end{equation*}
where we use the local smoothing estimate in Proposition \ref{prop:loc} again in the last inequality and we need $1-\nu_0<\beta<3/2$ such that $1/2<2-\beta<1+\nu_0$.
Therefore the above statement holds for all $\max\{1/2,1-\nu_0\}<\beta<3/2$.
By the Christ-Kiselev lemma \cite{CK}, thus we have shown that for $q>2$ and $(q,\rr)\in\Lambda_{s,\nu_0}=\Lambda_s$ with $s=\frac32-\beta$
\begin{equation}
\begin{split}
\eqref{est:inh}\lesssim \|u_0\|_{\dot H^{s}(X)}+\|u_1\|_{\dot H^{s-1}(X)}.
\end{split}
\end{equation}
 We remark that we have proved all  $(q,\rr)\in\Lambda_s$ with $q>2$ and $s$ such that $0<s<\min\{1,\frac12+\nu_0\}$.
 Now we relax the restriction to $s<\frac12+\nu_0$ when $\nu_0\geq 1/2$.
 For $1\leq s<\frac12+\nu_0$ and any $(q,\rr)\in\Lambda_s$, then there exists a pair $(q,\tilde{\rr})\in \Lambda_{\tilde{s}}$ with $\tilde{s}=1_-$ such that
 \begin{equation*}
\begin{split}
\| u(t,z)\|_{L^q(\R;L^{\rr}(X))}&\lesssim \|\LL_V^{(s-\tilde{s})/2} u(t,z)\|_{L^q(\R;L^{\tilde{\rr}}(X))}\\
&\lesssim \|u_0\|_{\dot H^{s}(X)}+\|u_1\|_{\dot H^{s-1}(X)}.
\end{split}
\end{equation*}
Indeed,  the Sobolev inequality of Corollary \ref{P:sobolev'} shows the first inequality and the above result implies the second one.

Therefore we have proved all  $(q,\rr)\in\Lambda_s$ and $s$ such that $0<s<\frac12+\nu_0$
except the endpoint admissible pair with $q=2$ when $s\geq s_0:=(n+1)/2(n-1)$ and $n\geq 4$.

 \vspace{0.2cm}

  {\bf Case II:} $0<\nu_0\leq \frac1{n-1}$. In this case, if $(q,\rr)\in\Lambda_{s,\nu_0}$, then $q>2$. Hence it suffices to fix the gap $\frac12+\nu_0\leq s<1+\nu_0$.
  To this end, we split the initial data into two parts: one is projected to $\mathcal{H}^\nu$ with $\nu\leq 1+\nu_0$ and the other is the remaining terms. Without loss of generality, we assume $u_1=0$ and
  divide $u_0=u_{0,l}+u_{0,h}$ where $u_{0,h}=u_0-u_{0,l}$ and
 \begin{equation}
 u_{0,l}=\sum_{\nu\in A}\sum_{\ell=1}^{d(\nu)}a_{\nu,\ell}(r)\varphi_{\nu,\ell}(y), \quad A=\{\nu\in\chi_\infty: \nu\leq 1+\nu_0\}.
 \end{equation}
 For the part involving $u_{0,h}$, we can repeat the argument of {\bf Case I}. In this case, as in Remark \ref{rem:loc-s},
 we can use Proposition \ref{prop:loc} with $1/2<2-\beta<2+\nu_0$.
 Thus we obtain the Strichartz estimate on $e^{it\sqrt{\LL_V}}u_{0,h}$ for $\Lambda_{s,\nu_0}$ with $s\in[\frac12+\nu_0,1+\nu_0)$. We remark here that the set $\Lambda_{s,\nu_0}$ is empty when $s\geq1+\nu_0$.

 Next we consider the Strichartz estimate for $e^{it\sqrt{\LL_V}}u_{0,l}$. We follow the argument of \cite{PST} which treated a radial case.
Recall
 \begin{equation}
\begin{split} e^{it\sqrt{\LL_V}}u_{0,l}&=
\sum_{\nu\in A}\sum_{\ell=1}^{d(\nu)}\varphi_{\nu,\ell}(y)\int_0^\infty(r\rho)^{-\frac{n-2}2}J_{\nu}(r\rho)e^{it\rho}\mathcal{H}_\nu(a_{\nu,\ell})\rho^{n-1}d\rho,\\
&=\sum_{\nu\in A}\sum_{\ell=1}^{d(\nu)}\varphi_{\nu,\ell}(y)\mathcal{H}_\nu[e^{it\rho}\mathcal{H}_\nu(a_{\nu,\ell})](r).
\end{split}
\end{equation}
Since $\nu\in A$, therefore there exists a constant $C_{\nu_0}$ depending on $\nu_0$ such that
 \begin{equation}
\begin{split} \|e^{it\sqrt{\LL_V}}u_{0,l}\|_{L^q(\R;L^\rr(X))}&\leq C_{\nu_0}
\sum_{\nu\in A}\sum_{\ell=1}^{d(\nu)}\left\|\mathcal{H}_\nu[e^{it\rho}\mathcal{H}_\nu(a_{\nu,\ell})](r)\right\|_{L^q(\R;L^\rr_{r^{n-1}dr})}.
\end{split}
\end{equation}
Let $\mu=(n-2)/2$ and recall $\mathcal{H}_\mu\mathcal{H}_\mu=Id$, then it suffices to estimate
 \begin{equation}
\begin{split}
\sum_{\nu\in A}\sum_{\ell=1}^{d(\nu)}\left\|(\mathcal{H}_\nu\mathcal{H}_\mu)\mathcal{H}_\mu[e^{it\rho}
\mathcal{H}_\mu(\mathcal{H}_\mu\mathcal{H}_\nu)(a_{\nu,\ell})](r)\right\|_{L^q(\R;L^\rr_{r^{n-1}dr})}.
\end{split}
\end{equation}
For our purpose, we recall \cite[Theorem 3.1]{PST} which claimed that the operator $\mathcal{K}^0_{\mu,\nu}:=\mathcal{H}_\mu\mathcal{H}_\nu$ is continuous on $L^p_{r^{n-1}dr}([0,\infty))$ if
$$\max\{((n-2)/2-\mu)/n,0\}<1/p<\min\{((n-2)/2+\nu+2)/n,1\}.$$
Notice $\mu=\frac{n-2}2$, on one hand, we have that both $\mathcal{K}^0_{\mu,\nu}$ and $\mathcal{K}^0_{\nu,\mu}$ are bounded in
$L^p_{r^{n-1}dr}([0,\infty))$ provided $\frac1p>\frac12-\frac{1+\nu}n$. One can check that $\frac1\rr>\frac12-\frac{1+\nu_0}n$
satisfies the condition since $\nu\geq \nu_0$. On the other hand,
$\mathcal{H}_\mu[e^{it\rho}\mathcal{H}_\mu]$ is a classical half-wave propagator in the radial case which has Strichartz estimate with $(q,\rr)\in\Lambda_s$.
In sum, for $(q,\rr)\in\Lambda_{s,\nu_0}$, we have
 \begin{equation}
\begin{split}
&\|e^{it\sqrt{\LL_V}}u_{0,l}\|_{L^q(\R;L^\rr(X))}\\&\leq C_{\nu_0}\sum_{\nu\in A}\sum_{\ell=1}^{d(\nu)}\left\|(\mathcal{H}_\nu\mathcal{H}_\mu)\mathcal{H}_\mu[e^{it\rho}\mathcal{H}_\mu(\mathcal{H}_\mu\mathcal{H}_\nu)(a_{\nu,\ell})](r)\right\|_{L^q(\R;L^\rr_{r^{n-1}dr})}\\
&\leq C_{\nu_0}\sum_{\nu\in A}\sum_{\ell=1}^{d(\nu)}\left\|(\mathcal{H}_\mu\mathcal{H}_\nu)(a_{\nu,\ell})](r)\right\|_{\dot H^s}\leq C_{\nu_0}\left(\sum_{\nu\in A}\sum_{\ell=1}^{d(\nu)}\left\|a_{\nu,\ell}(r)\right\|^2_{\dot H^s}\right)^{1/2}\leq C_{\nu_0}\|u_{0,l}\|_{\dot H^s}.
\end{split}
\end{equation}
In the second inequality, we use \cite[Theorem 3.8]{PST}.\vspace{0.2cm}

 {\bf Case III:} $\nu_0> \frac1{n-1}$, $q=2$ and $n\geq4$.
 In this case, we aim to prove
\begin{equation}\label{equ:en}
\begin{split}
\|u(t,z)\|_{L^2(\R;L^{\frac{2(n-1)}{n-3}}(X))}\lesssim \|u_0\|_{\dot H^{\frac{n+1}{2(n-1)}}(X)}+\|u_1\|_{\dot H^{\frac{n+1}{2(n-1)}-1}(X)}.
\end{split}
\end{equation}
Before proving this, we first prove
\begin{equation}\label{equ:ne}
\begin{split}
\|\LL_0^{-\frac1{n-1}}u(t,z)\|_{L^2(\R;L^{\frac{2(n-1)}{n-3}}(X))}
\lesssim \|u_0\|_{\dot H^{\frac{n-3}{2(n-1)}}(X)}+\|u_1\|_{\dot H^{\frac{n-3}{2(n-1)}-1}(X)}.
\end{split}
\end{equation}
Indeed, it follows from Corollary \ref{cor:inh} and Proposition \ref{prop:loc} with $\beta=(n-2)/(n-1)$ $(n\geq 4)$ that
\begin{equation*}
\begin{split}
&\Big\|\LL_0^{-\frac1{n-1}}\int_0^t\frac{\sin{(t-\tau)\sqrt{\LL_0}}}
{\sqrt{\LL_0}}(V(z)u(\tau,z))d\tau\Big\|_{L^2(\R;L^{\frac{2(n-1)}{n-3}}(X))}\\
&\lesssim \|V(z)u(\tau,z)\|_{L^{2}_t{L}^{\frac{2(n-1)}{n+1},2}_z}\lesssim \|r^{-\frac{n-2}{n-1}}u(\tau,z)\|_{L^{2}_t{L}^{2}_z}
\\&\lesssim \|u_0\|_{\dot H^{\frac{n-3}{2(n-1)}}(X)}+\|u_1\|_{\dot H^{\frac{n-3}{2(n-1)}-1}(X)}.
\end{split}
\end{equation*}
Hence this shows \eqref{equ:ne}. On the other hand, from Proposition \ref{P1}, we have shown that the operator
$$\LL_0^{\frac1{n-1}}\LL_V^{-\frac1{n-1}}:  L^{\frac{2(n-1)}{n+1}}(X)\to L^{\frac{2(n-1)}{n+1}}(X), \quad \nu_0>1/(n-1)$$
is bounded.
Note that the operators $\LL_V$ and $\LL_0$ are self-adjoint,  by dual argument, we see the boundedness of the operator
 $$\LL_V^{-\frac1{n-1}}\LL_0^{\frac1{n-1}}: L^{\frac{2(n-1)}{n-3}}(X)\to L^{\frac{2(n-1)}{n-3}}(X). $$
 Therefore we obtain
\begin{equation*}
\begin{split}
&\|u(t,z)\|_{L^2(\R;L^{\frac{2(n-1)}{n-3}}(X))}=\|\LL_V^{-\frac1{n-1}}\LL_0^{\frac1{n-1}}\LL_0^{-\frac1{n-1}}\LL_V^{\frac1{n-1}}u(t,z)\|_{L^2(\R;L^{\frac{2(n-1)}{n-3}}(X))}
\\&\lesssim \|\LL_V^{\frac1{n-1}} u_0\|_{\dot H^{\frac{n-3}{2(n-1)}}(X)}+\|\LL_V^{\frac1{n-1}} u_1\|_{\dot H^{\frac{n-3}{2(n-1)}-1}(X)}
\\&\lesssim \|u_0\|_{\dot H^{\frac{n+1}{2(n-1)}}(X)}+\|u_1\|_{\dot H^{\frac{n+1}{2(n-1)}-1}(X)}.
\end{split}
\end{equation*}
This gives \eqref{equ:en}.

Let $s_0=(n+1)/2(n-1)$
and apply the operator $\LL_V^{(s-s_0)/2}$ with $s_0\leq s<\frac12+\nu_0$
to the wave equation, thus by using the above Strichartz estimate, we obtain
\begin{equation*}
\begin{split}
\|\LL_V^{(s-s_0)/2} u(t,z)\|_{L^2(\R;L^{\frac{2(n-1)}{n-3}}(X))}\lesssim \|u_0\|_{\dot H^{s}(X)}+\|u_1\|_{\dot H^{s-1}(X)}.
\end{split}
\end{equation*}
Consider $(2,\frac{2n}{n-2s-1})\in \Lambda_s$  with $s_0\leq s<\frac12+\nu_0$. One can verify that $0\leq s-s_0<\min\{2, 1+\nu_0\}$ and
$\frac{2n}{n-2s-1}$ satisfies that \eqref{est:sobolev_hyp}.
By the Sobolev inequality in Corollary \ref{P:sobolev}, we show
\begin{equation*}
\begin{split}
\| u(t,z)\|_{L^2(\R;L^{\frac{2n}{n-2s-1}}(X))}&\lesssim \|\LL_V^{(s-s_0)/2} u(t,z)\|_{L^2(\R;L^{\frac{2(n-1)}{n-3}}(X))}\\
&\lesssim \|u_0\|_{\dot H^{s}(X)}+\|u_1\|_{\dot H^{s-1}(X)}.
\end{split}
\end{equation*}
In sum, under the condition $\nu_0>1/(n-1)$, we have proved the Strichartz estimate \eqref{stri} with $F=0$ for all  $(q,\rr)\in \Lambda_s$ with $s\in [0, 1/2+\nu_0)$.\vspace{0.2cm}

 {\bf Case VI:} $\nu_0> \frac1{n-1}$ and $(q,\rr)\in\Lambda_{s,\nu_0}$ with $s\in [1/2+\nu_0, 1+\nu_0)$. For $s\in [1/2+\nu_0, 1+\nu_0)$ and $(q,\rr)\in\Lambda_s\cap \{(q,\rr):\frac1\rr>\frac12-\frac{1+\nu_0}n\}$, as using the Sobolev inequality of Corollary \ref{P:sobolev'} as before,
 we have that there exists a pair $(q,\tilde{\rr})\in \Lambda_{\tilde{s}}$ with $\tilde{s}=(1/2+\nu_0)_-$ such that
 \begin{equation*}
\begin{split}
\| u(t,z)\|_{L^q(\R;L^{\rr}(X))}&\lesssim \|\LL_V^{(s-\tilde{s})/2} u(t,z)\|_{L^q(\R;L^{\tilde{\rr}}(X))}\\
&\lesssim \|u_0\|_{\dot H^{s}(X)}+\|u_1\|_{\dot H^{s-1}(X)}.
\end{split}
\end{equation*}
Thus we prove the homogeneous Strichartz estimate stated in Theorem \ref{thm:Strichartz}.
We show the inhomogeneous Strichartz estimate by using $TT^*$-method as in Subsection \ref{4.5}.  Therefore we complete the proof of the second conclusion in Theorem \ref{thm:Strichartz}.\vspace{0.2cm}

\subsection{The sharpness of the restriction \eqref{Ls}  }

In this subsection, we construct a counterexample to claim the restriction \eqref{Ls} is necessary for Theorem \ref{thm:Strichartz}.

\begin{proposition}[Counterexample]If $(q,\rr)\in \Lambda_s$ but  $(q,\rr)\notin \{(q,\rr):\frac1\rr>\frac12-\frac{1+\nu_0}n\}$. \vspace{0.2cm}
Then the Strichartz estimate possibly fails.

\end{proposition}

\begin{proof}  Assume $u_0=(\mathcal{H}_{\nu_0}\chi)(s)$ is independent of $y$, where $\chi\in\CC_c^\infty([1,2])$
is valued in $[0,1]$. Due to the compact support of  $\chi$ and the
unitarity of the Hankel transform $\mathcal{H}_{\nu_0}$ on $L^2$, we obtain $\|u_0\|_{\dot H^s}\leq C$. Now we conclude that
\begin{equation}
\begin{split} \|e^{it\sqrt{\mathcal{L}_V}}u_0\|_{L^q(\R;L^\rr(X))}=\infty\end{split}
\end{equation}
when $\frac1\rr\leq\frac12-\frac{1+\nu_0}n$. We write that
\begin{equation}
\begin{split} e^{it\sqrt{\mathcal{L}_V}}u_0&=\int_0^\infty(s\rho)^{-\frac{n-2}2}J_{\nu_0}(s\rho)e^{
it\rho}(\mathcal{H}_{\nu_0}u_0)(\rho)\rho^{n-1} d\rho\\
&=\int_0^\infty(s\rho)^{-\frac{n-2}2}J_{\nu_0}(s\rho)e^{
it\rho}\chi(\rho)\rho^{n-1}d\rho.
\end{split}
\end{equation}
We recall the behavior of  $J_\nu(r)$ as $r\to 0+$. For the complex number $\nu$ with $\mathrm{Re}(\nu)>-1/2$, see \cite[Section B.6]{G}, then we have that
\begin{equation}\label{Bessel1}
J_{\nu}(r)=\frac{r^\nu}{2^\nu\Gamma(\nu+1)}+S_\nu(r)
\end{equation}
where
\begin{equation}\label{Bessel2}
S_{\nu}(r)=\frac{(r/2)^{\nu}}{\Gamma\left(\nu+\frac12\right)\Gamma(1/2)}\int_{-1}^{1}(e^{isr}-1)(1-s^2)^{(2\nu-1)/2}ds
\end{equation}
satisfies
\begin{equation}\label{Bessel3}
|S_{\nu}(r)|\leq \frac{2^{-\mathrm{Re}\nu}r^{\mathrm{Re}\nu+1}}{(\mathrm{Re}\nu+1)|\Gamma(\nu+\frac12)|\Gamma(\frac12)}.\end{equation}
Now we compute for any $0<\epsilon\ll 1$
\begin{align*}
\|e^{it\sqrt{\mathcal{L}_V}}u_0\|_{L^q(\R;L^\rr)}
=&\left\|\int_0^\infty(s\rho)^{-\frac{n-2}2}J_{\nu_0}(s\rho)e^{
it\rho}\chi(\rho)\rho^{n-1}d\rho\right\|_{L^q(\R;L^\rr)}\\
\geq&\left\|\int_0^\infty(s\rho)^{-\frac{n-2}2}J_{\nu_0}(s\rho)e^{
it\rho}\chi(\rho)\rho^{n-1}d\rho\right\|_{L^q([0,1/2];L^\rr_{s^{n-1}ds}[\epsilon,1])}\\
\geq& c\left\|\int_0^\infty(s\rho)^{-\frac{n-2}2}(s\rho)^{\nu_0}e^{
it\rho}\chi(\rho)\rho^{n-1}d\rho\right\|_{L^q([0,1/2];L^\rr_{s^{n-1}ds}[\epsilon,1])}\\
&-\left\|\int_0^\infty(s\rho)^{-\frac{n-2}2}S_{\nu_0}(s\rho)e^{
it\rho}\chi(\rho)\rho^{n-1}d\rho\right\|_{L^q([0,1/2];L^\rr_{s^{n-1}ds}[\epsilon,1])}.
\end{align*}
We first observe that by \eqref{Bessel3}
\begin{equation}
\begin{split}
&\left\|\int_0^\infty(s\rho)^{-\frac{n-2}2}S_{\nu_0}(s\rho)e^{
it\rho}\chi(\rho)\rho^{n-1}d\rho\right\|_{L^q([0,1/2];L^\rr_{s^{n-1}ds}[\epsilon,1])}\\
\leq& C\left\|\int_0^\infty(s\rho)^{-\frac{n-2}2}(s\rho)^{\nu_0+1}\chi(\rho)\rho^{n-1}d\rho\right\|_{L^q([0,1/2];L^\rr_{s^{n-1}ds}[\epsilon,1])}\\
\leq& C\max\big\{\epsilon^{\nu_0+1-\frac{n-2}2+\frac n\rr},1\big\}.
\end{split}
\end{equation}
Next we estimate the lower boundedness
\begin{align*}
&\left\|\int_0^\infty(s\rho)^{-\frac{n-2}2}(s\rho)^{\nu_0}e^{
it\rho}\chi(\rho)\rho^{n-1}d\rho\right\|_{L^q([0,1/4];L^\rr_{s^{n-1}ds}[\epsilon,1])}\\
=&C_{\nu_0}\left(\int_0^{\frac14}\left( \int_{\epsilon}^1 \left|\int_0^\infty(s\rho)^{-\frac{n-2}2}(s\rho)^{\nu_0}e^{
it\rho}\chi(\rho)\rho^{n-1}d\rho\right|^\rr s^{n-1}ds\right)^{q/\rr}dt\right)^{1/q}\\
\geq&C_{\nu_0}\left(\int_0^{\frac14}\left|\int_0^\infty\rho^{-\frac{n-2}2}\rho^{\nu_0}e^{
it\rho}\chi(\rho)\rho^{n-1}d\rho\right|^{q}dt\right)^{1/q}\times \begin{cases}\epsilon^{\nu_0-\frac{n-2}2+\frac n\rr} \quad\text{if}\quad \frac1\rr<\frac12-\frac{\nu_0+1}{n}\\
\ln\epsilon \quad\text{if}\quad \frac1\rr=\frac12-\frac{\nu_0+1}{n}
\end{cases}\\
\geq& c\begin{cases}\epsilon^{\nu_0-\frac{n-2}2+\frac n\rr} \quad\text{if}\quad \frac1\rr<\frac12-\frac{\nu_0+1}{n};\\
\ln\epsilon \quad\text{if}\quad \frac1\rr=\frac12-\frac{\nu_0+1}{n}
\end{cases}
\end{align*}
where we have used the fact that $\cos(\rho t)\geq 1/2$ for $t\in [0, 1/4]$ and $\rho\in [1,2]$, and
\begin{equation}
\begin{split}
\left|\int_0^\infty\rho^{-\frac{n-2}2}\rho^{\nu_0}e^{
it\rho}\chi(\rho)\rho^{n-1}d\rho\right|\geq \frac12\int_0^\infty\rho^{-\frac{n-2}2}\rho^{\nu_0}\chi(\rho)\rho^{n-1}d\rho\geq c.
\end{split}
\end{equation}
Hence, we obtain if $\frac1\rr<\frac12-\frac{\nu_0+1}n$
\begin{equation}
\begin{split}
\|e^{it\sqrt{\mathcal{L}_V}}u_0\|_{L^q(\R;L^\rr)}&\geq c\epsilon^{\nu_0-\frac{n-2}2+\frac n\rr}-C\max\big\{\epsilon^{\nu_0+1-\frac{n-2}2+\frac n\rr},1\big\}
\\&\geq c\epsilon^{\nu_0-\frac{n-2}2+\frac n\rr}\to +\infty \quad \text{as}\quad \epsilon\to 0
\end{split}
\end{equation}
 And when $\frac1\rr=\frac12-\frac{\nu_0+1}n$, we get
\begin{align*}
\|e^{it\sqrt{\mathcal{L}_V}}u_0\|_{L^q(\R;L^\rr)}&\geq c\ln\epsilon-C\to +\infty \quad \text{as}\quad \epsilon\to 0.
\end{align*}

\end{proof}

\section{Applications: well-posedness and scattering theory }
In this section, we prove Theorem \ref{thm:NLW} by using the Strichartz estimates established in Theorem \ref{thm:Strichartz}.
We follow the standard Banach fixed point argument to prove this
result. For any small constant $\epsilon>0$, let $I=[0,T)$, there exists $T>0$
such that
\begin{align*}
B_I:=\big\{&u\in C(I,\dot{H}^1):~\|(u,\partial_t u)\|_{\dot{H}^1\times L^2}\leq 2C\|(u_0,u_1)\|_{\dot{H}^1\times L^2},
\|u\|_{L_t^qL_z^{\rr}(I\times X)}\leq 2C\epsilon\big\}.
\end{align*}
To this end, we consider the map
\begin{equation}\label{inte3}
\Phi(u(t))=\dot{K}(t)u_0+K(t)u_1+\int_{0}^tK(t-s)(|u|^{\frac{4}{n-2}}u(s))ds
\end{equation}
on the complete metric space $B_I$
with the metric
$d(u,v)=\big\|u-v\big\|_{L_t^{q}L_z^{\rr}(I\times X)}$ and where the pair $(q,\rr)$ is given by \eqref{q-r}. We can check that $(q,\rr)\in \Lambda_1$.
On the other hand, we observe that
if the initial data has small enough size $\delta$, then by Strichartz estimate
\begin{equation}\label{small}
\big\|\dot{K}(t)u_0+K(t)u_1\big\|_{L_t^{q}L_z^{\rr}(I\times X)}\leq C \epsilon
\end{equation}
holds for $T=\infty$; if not, the inequality holds for some small $T>0$ by the dominated convergence theorem.
We need to
prove that the operator $\Phi$ defined by \eqref{inte3}
is well-defined on $B_I$ and is a contraction map under the metric $d$
for $I$.\vspace{0.1cm}

Let $u\in B_I$ with $0<\epsilon\ll 1$.
We first consider $3\leq n\leq 6$. Then, we have by Strichartz estimate
\begin{align*}
&\|\Phi(u(t))\|_{L_t^{(n+2)/(n-2)}L_z^{2(n+2)/(n-2)}(I\times X)}\\ &\leq\big\|\dot{K}(t)u_0+K(t)u_1\big\|_{L_t^{(n+2)/(n-2)}L_z^{2(n+2)/(n-2)}(I\times X)}+C\big\||u|^{\frac4{n-2}}u\big\|_{L_t^1L_z^2(I\times X)}\\
&\leq C\epsilon+\|u\|_{L_t^{(n+2)/(n-2)}L_z^{2(n+2)/(n-2)}(I\times X)}^{\frac{n+2}{n-2}}
\leq 2 C\epsilon,
\end{align*}
and
\begin{align*}
\sup_{t\in I}\big\|(\Phi(u),\partial_t\Phi(u))\big\|_{\dot{H}^1\times L^2}\leq&C\big\|(u_0,u_1)\|_{\dot{H}^1\times L^2}+C\big\||u|^{\frac4{n-2}}u\big\|_{L_t^1L_z^2(I\times X)}\\
\leq&2C\big\|(u_0,u_1)\|_{\dot{H}^1\times L^2}.
\end{align*}
Next we consider the case $n\geq 7$. By using the Strichartz estimate again, we show
\begin{align*}
&\|\Phi(u(t))\|_{L_t^{2}L_z^{2n/(n-3)}(I\times X)}\\ &\leq\big\|\dot{K}(t)u_0+K(t)u_1\big\|_{L_t^{2}L_z^{2n/(n-3)}}+C\big\||u|^{\frac4{n-2}}u\big\|_{L_t^{2(n-2)/(n+2)}L_z^{2n(n-2)/(n-3)(n+2)}}\\
&\leq C\epsilon+\|u\|_{L_t^2 L_z^{2n/(n-3)}(I\times X)}^{\frac{n+2}{n-2}}
\leq 2C\epsilon,
\end{align*}
and
\begin{align*}
&\sup_{t\in I}\big\|(\Phi(u),\partial_t\Phi(u))\big\|_{\dot{H}^1\times L^2}\\ &\leq C\big\|(u_0,u_1)\|_{\dot{H}^1\times L^2}+C\big\||u|^{\frac4{n-2}}u\big\|_{L_t^{2(n-2)/(n+2)}L_z^{2n(n-2)/(n-3)(n+2)}(I\times X)}\\
&\leq2C\big\|(u_0,u_1)\|_{\dot{H}^1\times L^2}.
\end{align*}
Hence for $n\geq3$ we have $\Phi(u)\in B_I$. On the other hand,  for $\omega_1, \omega_2\in B_I$, by Strichartz
estimate and choosing $\epsilon$ sufficiently small, we obtain for $3\leq n\leq 6$
\begin{align*}
&d\big(\Phi(w_1),\Phi(w_2)\big)\\ &\leq C\big\||w_1|^{\frac{4}{n-2}}w_1-|w_2|^{\frac{4}{n-2}}w_2\big\|_{L_t^1L_z^2(I\times X)}\\
&\leq C\|w_1-w_2\|_{L_t^{(n+2)/(n-2)}L_z^{2(n+2)/(n-2)}}\big(\|w_1\|_{L_t^{(n+2)/(n-2)}L_z^{2(n+2)/(n-2)}}^{\frac{4}{n-2}}+\|w_2\|_{L_t^{(n+2)/(n-2)}L_z^{2(n+2)/(n-2)}}^{\frac{4}{n-2}}\big)\\
&\leq \tilde{C}\epsilon^{\frac{4}{n-2}}d(w_1,w_2)\leq\tfrac12d(w_1,w_2),
\end{align*}
and for $n\geq 7$
\begin{align*}
&d\big(\Phi(w_1),\Phi(w_2)\big)\\ &\leq C\big\||w_1|^{\frac{4}{n-2}}w_1-|w_2|^{\frac{4}{n-2}}w_2\big\|_{L_t^{2(n-2)/(n+2)}L_z^{2n(n-2)/(n-3)(n+2)}}\\
&\leq C\|w_1-w_2\|_{L_t^{2}L_z^{2n/(n-3)}}\big(\|w_1\|_{L_t^{2}L_z^{2n/(n-3)}}^{\frac{4}{n-2}}+\|w_2\|_{L_t^{2}L_z^{2n/(n-3)}}^{\frac{4}{n-2}}\big)\\
&\leq \tilde{C}\epsilon^{\frac{4}{n-2}}d(w_1,w_2)\leq\tfrac12d(w_1,w_2),
\end{align*}
The standard fixed point argument  gives a unique solution $u$ of
\eqref{equ:energy-critical} on $I\times X$ which satisfies the bound
\eqref{local-b}. Therefore if $\delta$ is small enough, we obtain the global solution; otherwise, we have the local existence.\vspace{0.1cm}

Next, we turn to show the scattering result.
We just prove that $u$ scatters at $+\infty$, the proof for the
scattering at $-\infty$ is similar.  Using Duhaml's formula, the
solution with initial data $(u(0),\dot{u}(0))=(u_0,u_1)\in \dot{H}^1\times
L^2$ of \eqref{equ:energy-critical} can be written as
\begin{equation}
{u(t)\choose\dot{u}(t)}=V_0(t){u_0 \choose
u_1}-\int^{t}_{0}V_0(t-s){0\choose F(u(s))}ds,
\end{equation}
where $V_0$ is defined by \eqref{v0tdefin}. Denote the scattering data
$(u_0^+,u_1^+)$ by
$${u_0^+\choose u_1^+}={u_0\choose u_1}-\int_0^\infty V_0(-s){0\choose F(u(s))}ds.$$
Then, by Strichartz estimate,  we can
obtain for $3\leq n\leq 6$
\begin{align*}
\Big\|{u\choose \dot{u}}-V_0(t){u_0^+\choose u_1^+}\Big\|_{\dot H^1\times L^2}&=\Big\|\int_t^\infty V_0(t-s){0\choose F(u(s))}ds\Big\|_{\dot H^1\times L^2}\\
&\lesssim\|(|u|^{\frac{4}{n-2}})u\|_{L_t^1L_z^2((t,\infty)\times X)}\\
&\lesssim\|u\|_{L_t^{(n+2)/(n-2)}L_z^{2(n+2)/(n-2)}((t,\infty)\times X)}^{\frac{4}{n-2}}\\
&\rightarrow 0,\quad \text{as}\quad t\rightarrow\infty.
\end{align*}
and for $n\geq7$
\begin{align*}
\Big\|{u\choose \dot{u}}-V_0(t){u_0^+\choose u_1^+}\Big\|_{\dot H^1\times L^2}&=\Big\|\int_t^\infty V_0(t-s){0\choose F(u(s))}ds\Big\|_{\dot H^1\times L^2}\\
&\lesssim\|(|u|^{\frac{4}{n-2}})u\|_{L_t^{2(n-2)/(n+2)}L_z^{2n(n-2)/(n-3)(n+2)}}\\
&\lesssim\|u\|_{L_t^{2}L_z^{2n/(n-3)}((t,\infty)\times X)}^{\frac{4}{n-2}}\\
&\rightarrow 0,\quad \text{as}\quad t\rightarrow\infty.
\end{align*}
Thus we prove that $u$ scatters.

\begin{center}

\end{center}
\end{document}